\documentclass[a4paper]{amsart}

\usepackage[utf8]{inputenc}
\usepackage{amsmath}
\usepackage{graphicx}
\usepackage{enumitem}
\usepackage{pdfpages}
\usepackage{amssymb}

\PassOptionsToPackage{cmyk}{xcolor}
\usepackage{tikz}
\usetikzlibrary{decorations.pathreplacing}
\usetikzlibrary{decorations.markings}
\usetikzlibrary{calc}
\tikzset{->-/.style={decoration={
  markings,
  mark=at position 0.5 with {\arrow{>}}},postaction={decorate}}}
\usepackage{subfig}

\newcounter{stefancomments}

\makeatletter
\newcommand{\mlabel}[2]{(#2)\def\@currentlabel{#2}\label{#1}}
\makeatother

\usepackage{empheq}
\numberwithin{equation}{section}

\usepackage{amsthm}
\theoremstyle{plain}
\newtheorem*{theorem*}{Theorem}
\newtheorem{introtheorem}{Theorem}

\newtheorem{theorem}{Theorem}[section]
\newtheorem{lemma}[theorem]{Lemma}
\newtheorem{corollary}[theorem]{Corollary}
\newtheorem{proposition}[theorem]{Proposition}
\newtheorem{observation}[theorem]{Observation}

\theoremstyle{definition}

\newtheorem{remark}[theorem]{Remark}
\newtheorem{example}[theorem]{Example}

\newcommand{\N}{\mathbb{N}}

\newcommand{\Z}{\mathbb{Z}}

\newcommand{\defeq}{\mathrel{\vcentcolon =}}
\newcommand{\cat}{\mathcal{C}}
\newcommand{\precat}{\mathcal{S}}
\newcommand{\catrels}{\mathcal{R}}
\newcommand{\grpd}{\mathcal{G}}
\newcommand{\thomcat}{\mathcal{F}}
\newcommand{\Tcat}{\mathcal{T}}
\newcommand{\Vcat}{\mathcal{V}}
\newcommand{\BVcat}{\mathcal{BV}}
\newcommand{\BTcat}{\mathcal{BT}}
\newcommand{\BFcat}{\mathcal{BF}}

\newcommand{\rewcat}{\mathcal{R}}
\newcommand{\rewgpcat}{\mathcal{RG}}
\newcommand{\Ob}{\mathrm{Ob}}
\newcommand{\Sym}{\textsc{Sym}}
\newcommand{\Braid}{\textsc{Braid}}

\newcommand{\Homeo}{\operatorname{Homeo}}
\newcommand{\lk}{\operatorname{lk}}
\newcommand{\st}{\operatorname{st}}
\newcommand{\tail}{\operatorname{tail}}
\newcommand{\head}{\operatorname{head}}
\newcommand{\Div}{\operatorname{Div}}
\newcommand{\tDiv}{\widetilde{\operatorname{Div}}}
\newcommand{\match}{\mathcal{M}}
\newcommand{\arcmatch}{\mathcal{MA}}
\newcommand{\env}{\mathcal{G}pd}
\newcommand{\Ore}{\mathcal{O}re}

\newcommand{\eltry}{\mathcal{E}}
\newcommand{\gars}{\mathcal{S}}

\newcommand{\gen}[1]{\left\langle#1\right\rangle}
\newcommand{\abs}[1]{\lvert #1 \rvert}

\newcommand{\floor}[1]{\lfloor #1 \rfloor}
\newcommand{\from}{\leftarrow}

\title[Classifying spaces, Ore categories, Garside families]{Classifying spaces from\\Ore categories with Garside families}
\date{\today}



\author[S.~Witzel]{Stefan Witzel}
 \address{Faculty of Mathematics, Bielefeld University, Postfach 100131, 33501 Bielefeld, Germany}
 \email{switzel@math.uni-bielefeld.de}

\tikzstyle{vertex} = [fill, shape = circle, inner sep=0pt,minimum size=2pt]
\tikzstyle{end} = [fill, shape = circle, inner sep=0pt,minimum size=2pt]

\begin{document}

\newcommand{\picangle}{60}
\newcommand{\loopangle}{40}

\begin{abstract}
We describe how an Ore category with a Garside family can be used to construct a classifying space for its fundamental group(s). The construction simultaneously generalizes Brady's classifying space for braid groups and the Stein--Farley complexes used for various relatives of Thompson's groups. It recovers the fact that Garside groups have finite classifying spaces.

We describe the categories and Garside structures underlying certain Thompson groups. The Zappa--Szép product of categories is introduced and used to construct new categories and groups from known ones. As an illustration of our methods we introduce the group Braided $T$ and show that it is of type $F_\infty$.
\end{abstract}

\maketitle

Our main object of study are groups that arise as the fundamental group of an Ore category with a Garside family. The two basic motivating examples are the braid groups $\Braid_n$ and Thompson's group $F$. We provide tools to construct classifying spaces with good finiteness properties for these groups. Our first main result can be formulated as follows (see Section~\ref{sec:ore_cats} for definitions and Section~\ref{sec:finiteness_properties} for the general version).

\begin{introtheorem}
\label{thm:main_complex}
Let $\cat$ be a small right-Ore category that is factor-finite, let $\Delta$ be a right-Garside map, and let $* \in \Ob(\cat)$. There is a contractible simplical complex $X$ on which $G = \pi_1(\cat,*)$ acts. The space is covered by the $G$-translates of compact subcomplexes $K_x, x \in \Ob(\cat)$. Every stabilizer is isomorphic to a finite index subgroup of $\cat^\times(x,x)$ for some $x \in \cat$.
\end{introtheorem}

Taking $\cat$ to be a Garside monoid and $\Delta$ to be the Garside element, one immediately recovers the known fact that Garside groups, and braid groups in particular, have finite classifying spaces \cite{charney04}. In fact, if $\cat$ is taken to be the dual braid monoid, the quotient $G \backslash X$ is precisely Brady's classifying space for $\Braid_n$ \cite{brady01}.

In the case of Thompson's group $F$ the complex in Theorem~\ref{thm:main_complex} is the Stein--Farley complex. The action is not cocompact in this case because $\cat$ has infinitely many objects. In order to obtain cocompact actions on highly connected spaces, we employ Morse theory.

\begin{introtheorem}
\label{thm:main_fn}
Let $\cat$, $\Delta$, $*$ be as in Theorem~\ref{thm:main_complex} and let $\rho \colon \Ob(\cat) \to \N$ be a height function such that $\{x \in \Ob(\cat) \mid \rho(x) \le n\}$ is finite for every $n \in \N$. Assume that
\begin{enumerate}[align=left,leftmargin=*,widest=(\textsc{stab})]
\item[(\textsc{stab})] $\cat^\times(x,x)$ is of type $F_n$ for all $x$,
\item[\mlabel{item:link_connectivity}{\textsc{lk}}] $\abs{E(x)}$ is $(n-1)$-connected for all $x$ with $\rho(x)$ beyond a fixed bound.
\end{enumerate}
Then $\pi_1(\cat,*)$ is of type $F_n$.
\end{introtheorem}

The methods for proving Theorem~\ref{thm:main_fn} have been repeatedly used to show that various Thompson groups are of type $F_\infty$ \cite{brown87,stein92,farley03,fluch13,bux16,witzel16a,martinezperez16,belk}. Theorem~\ref{thm:main_fn} could be used as a drop-in replacement in most of the proofs. The complexes $\abs{E(x)}$ depend on $\cat$ and $\Delta$ and are described in Section~\ref{sec:proof_scheme}. In general, verifying condition \eqref{item:link_connectivity} is the key problem in all the $F_\infty$ proofs mentioned.

Theorem~\ref{thm:main_fn} provides a general scheme for proving that an (eligible) group is of type $F_\infty$: first describe the category, second analyze the complexes $\abs{E(x)}$, and then apply the theorem. This scheme will be illustrated in Section~\ref{sec:indirect_product_examples} (describe the category) and Section~\ref{sec:finiteness_properties_examples} (analyze the complexes, apply the theorem) on the examples of Thompson's groups $F$, $T$ and $V$, their braided versions and some other groups. To our knowledge this is the first time that Garside structures are studied in connection with Thompson groups. In the course we define the Thompson group $\mathit{BT}$, \emph{braided $T$}, and prove (see Theorem~\ref{thm:braided_finiteness_properties}):

\begin{introtheorem}
\label{thm:bt}
The braided Thompson group $\mathit{BT}$ is of type $F_\infty$.
\end{introtheorem}

Although braided versions of $V$ \cite{dehornoy06,brin07} and $F$ \cite{brady08} exist in the literature, our main merit is to be able to define braided $T$. The fact that it is $F_\infty$ then follows from Theorem~\ref{thm:main_fn} and results from \cite{bux16}. The reason that $\mathit{BT}$ was not defined before is that the natural categorical approach was artificially and painfully avoided in the past, see Remark~\ref{rem:bt}.

A helpful tool in describing the needed categories is the Zappa--Szép $\thomcat \bowtie \grpd$ product of two categories $\thomcat,\grpd$. We call it the \emph{indirect product} and introduce it in Section~\ref{sec:indirect_product}.

The article is organized as follows. The basic notions are introduced in Section~\ref{sec:ore_cats}. The underlying structures for braid groups and Thompson's group $F$ are described in Section~\ref{sec:f}. Section~\ref{sec:finiteness_properties} contains the main construction and the proofs of Theorems~\ref{thm:main_complex} and~\ref{thm:main_fn}. The indirect product of categories is introduced in Section~\ref{sec:indirect_product} and is used in Section~\ref{sec:indirect_product_examples} to construct the categories underlying Thompson's groups and their braided versions. In Section~\ref{sec:finiteness_properties_examples} Theorem~\ref{thm:main_fn} is applied to the examples from Section~\ref{sec:indirect_product_examples} to deduce finiteness properties, among them Theorem~\ref{thm:bt}. Since the results about finiteness properties and the indirect product may be of independent interest, we include the following leitfaden.

\begin{center}
\tikzstyle{box}=[draw, rectangle, rounded corners]
\begin{tikzpicture}[scale=1.5]
\draw node [box] at (0,0) (basic) {Section 1};
\draw node [box] at (-2,-1) (basicex) {Section 2};
\draw node [box] at (0,-1) (indirect) {Section 4};
\draw node [box] at (2,-1) (fin) {Section 3};
\draw node [box] at (-1,-2) (indirectex) {Section 5};
\draw node [box] at (1,-2.5) (finex) {Section 6};
\draw[->] (basic) -- (basicex);
\draw[->] (basic) -- (indirect);
\draw[->]  (indirect) -- (indirectex);
\draw[->]  (basicex) -- (indirectex);
\draw[->]  (basic) -- (fin);
\draw[->]  (fin) -- (finex);
\draw[->]  (indirectex) -- (finex);
\end{tikzpicture}
\end{center}

This article arose out of the introduction to the author's habilitation thesis \cite{witzel16b} which contains further examples not covered here.

\section{Categories generalizing monoids}
\label{sec:ore_cats}

We start by collecting basic notions of categories regarding them as generalizations of monoids. Our exposition is based on \cite[Chapter~II]{dehornoy15} where the perspective is similar. The main difference is notational, see Remark~\ref{rem:dehornoy_compatibility} below.

A monoid may be regarded as (the set of morphisms) of a category with a single object. For us categories will play the role of generalized monoids where the multiplication is only partially defined. In particular, all categories in this chapter will be small. The requirement that they be locally small is important and taking them to be small is convenient, for example, it allows us to talk about morphisms of categories as maps of sets.

Let $\cat$ be a category. Notationally, we follow \cite{dehornoy15} in denoting the set of morphisms of $\cat$ by $\cat$ as well (thinking of them as elements), while the objects are denoted $\Ob(\cat)$. The identity at $x$ will be denoted $1_x$. If $f$ is a morphism from $y$ to $x$ we call $y$ the \emph{source} and $x$ the \emph{target} of $f$. Our notation for composition is is the familiar one for functions, that is, if $f$ is a morphism from $y$ to $x$ and $g$ is a morphism from $z$ to $y$ then $fg$ exists and is a morphism from $z$ to $x$. If $x, y \in \Ob(\cat)$ then the set of morphisms from $y$ to $x$ is denoted $\cat(x,y)$, the set of morphisms from $y$ to any object is denoted $\cat(-,y)$ and the set of morphisms from any object to $x$ is denoted $\cat(x,-)$. This may be slightly unusual but renders the following intuitive expression valid:
\[
f \in \cat(x,y), g \in \cat(y,z) \Rightarrow fg \in \cat(x,z)\text{.}
\]
The corresponding diagram is
\begin{center}
\begin{tikzpicture}[scale=1.2]
\node (x) at (0,0) {$x$};
\node (y) at (1,0) {$y$};
\node (z) at (2,0) {$z.$};
\path[<-]
(x) edge node[above]{$f$} (y)
(y) edge node[above]{$g$} (z);
\path[<-]
(x) edge[out=-30,in=-150] node[below]{$fg$} (z);
\end{tikzpicture}
\end{center}
When we write an expression involving a product of morphisms, the requirement that this product exists is usually an implicit condition of the expression. Thus $fg = h$ means that the source of $f$ is the target of $g$ and that the equality holds.

\begin{remark}
\label{rem:dehornoy_compatibility}
The net effect of the various differences in notation is that our formalism is consistent with \cite{dehornoy15}, only the meaning of source/target, from/to, and the direction of arrows are switched. The reason for this decision is that some of our morphisms will be group elements which we want to act from the left.
\end{remark}

\subsection{Groupoids}

A morphism $f \in \cat(x,y)$ is \emph{invertible} if there is an \emph{inverse}, namely a morphism $g \in \cat(y,x)$ such that $fg = 1_x$ and $gf = 1_y$.  A \emph{groupoid} is a category in which every morphism is invertible. Just as every monoid naturally maps to a group, every category naturally maps to a groupoid, see \cite[Section~3.1]{dehornoy15}:

\begin{theorem}
For every category $\cat$ there is a groupoid $\env(\cat)$ and a morphism $\iota \colon \cat \to \env(\cat)$ with the following universal property: if $\varphi \colon \cat \to \grpd$ is a morphism to a groupoid then there is a unique morphism $\hat{\varphi} \colon \env(\cat) \to \grpd$ such that $\varphi = \hat{\varphi} \circ \iota$.

The groupoid $\env(\cat)$ and the morphism $\iota$ are determined by $\cat$ uniquely up to unique isomorphism.
\end{theorem}

We call $\env(\cat)$ the \emph{enveloping groupoid} of $\cat$. The morphism 
$\iota$ is a bijection on objects but it is not typically injective (on 
morphisms). One way to think about the enveloping groupoid is as the 
fundamental groupoid of $\cat$:

The \emph{nerve} of $\cat$ is the simplicial set whose $k$-simplices are diagrams
\begin{center}
\begin{tikzpicture}[scale=1.2]
\node (z) at (0,0) {$x_0$};
\node (y) at (1,0) {$x_1$};
\node (x) at (2,0) {$x_2$};
\node (w) at (3,0) {$x_{k-1}$};
\node (v) at (4,0) {$x_k$};
\path[<-]
(z) edge node[above]{$f_1$} (y)
(y) edge node[above]{$f_2$} (x)
(w) edge node[above]{$f_k$} (v);
\path[dotted]
(x) edge (w);
\end{tikzpicture}
\end{center}
in $\cat$. The $i$th face is obtained by deleting $x_i$ and replacing $f_i, f_{i+1}$ by $f_i f_{i+1}$ and the $j$th degenerate coface is obtained by introducing $1_{x_j}$ between $f_j$ and $f_{j+1}$.

\begin{proposition}[{\cite[Proposition~1]{quillen73}}]
The groupoid $\env(\cat)$ is canonically isomorphic to the fundamental groupoid of the realization of the nerve of $\cat$.
\end{proposition}

In particular, the fundamental group of $\cat$ in an object $x$ is just the set of endomorphisms of $\env(\cat)$ in $x$: $\pi_1(\cat,x) = \env(\cat)(x,x)$.

\subsection{Noetherianity conditions}

If $fg = h$ then we say that $f$ is a \emph{left-factor} of $h$ and that $h$ is a \emph{right-multiple} of $f$. It is a \emph{proper left-factor} respectively \emph{proper right-multiple} if $g$ is not invertible. We say that $f$ is a \emph{(proper) factor} of $h$ if $efg = h$ (and one of $e$ and $g$ is not invertible).

The category $\cat$ is \emph{Noetherian} if there is no infinite sequence $f_0, f_1, \ldots$ such that $f_{i+1}$ is a proper factor of $f_i$. It is said to be \emph{strongly Noetherian} if there exists a map $\delta \colon \cat \to \N$ that satisfies $\delta(fg) \ge \delta(f) + \delta(g)$ and for $f \in \cat$ non-invertible $\delta(f) \ge 1$. Clearly, a strongly Noetherian category is Noetherian. See \cite[Sections~II.2.3, II.2.4]{dehornoy15} for a detailed discussion.

We call a \emph{height function} a map $\rho \colon \Ob(\cat) \to \N$ such that $\rho(x) = \rho(y)$ if $\cat(x,y)$ contains an invertible morphism and $\rho(x) < \rho(y)$ if $\cat(x,y)$ contains a non-invertible morphism. Note that the existence of a height function implies strong Noetherianity by taking $\delta(f) = \rho(y) - \rho(x)$ if $f \in \cat(x,y)$.

We say that $\cat$ is \emph{factor-finite} if every morphism in $\cat$ has only finitely many factors up pre- and post-composition by invertibles. This condition implies strong Noetherianity (cf.\ \cite[Proposition~2.48]{dehornoy15}).

\subsection{Ore categories}
\label{sec:ore_property}

Two elements $g,h \in \cat(x,-)$ have a \emph{common right-multiple} $d$ if there 
exist elements $e,f \in \cat$ with $ge = hf = d$. It is a \emph{least common 
right-multiple} if every other common right-multiple is a right-multiple of $d$. We 
say that $\cat$ \emph{has common right-multiples} if any two elements with same 
target have a common right-multiple. We say that it \emph{has conditional least 
common right-multiples} if any two elements that have a common right-multiple 
have a least common right-multiple. We say that it \emph{has least common right-
multiples} if any two elements with same target have a least common right-
multiple. We say that $\cat$ is \emph{left-cancellative} if $ef = eh$ implies 
$f=h$ for all $e,f,g \in \cat$. All of these notions have obvious analogues with 
left and right interchanged. A category is \emph{cancellative} if it is 
left-cancellative and right-cancellative.

\begin{lemma}
\label{lem:cancellative_inverse}
If $\cat$ is cancellative and $f \in \cat$ has a left-inverse or right-inverse then it is invertible.
\end{lemma}

\begin{proof}
Let $f \in \cat(x,y)$ and assume that there is an $e \in \cat(y,x)$ that is a 
left-inverse for $f$, that is, $ef = 1_y$. Then $fef = f$ and canceling $f$ on 
the right shows that $e$ is also a right-inverse. The other case is symmetric.
\end{proof}

\begin{lemma}
\label{lem:gcd_vs_lcm}
Let $\cat$ be strongly Noetherian. Then $\cat$ has least common right-multiples if and only if it has greatest common left-factors.
\end{lemma}

\begin{proof}
Suppose that $\cat$ has least common right-multiples and let $f,g \in \cat(x,-)$. Let $s$ and $t$ be common left-factors of $f,g$ and let $r$ be a least common right-multiple of $s$ and $t$. Then, since $f$ and $g$ are common right-multiples of $s$ and $t$, they are right-multiples of $r$, meaning that $r$ is a common left-factor. If $s$ and $t$ are not right-multiples of each other then $\delta(r) > \delta(s), \delta(t)$ and an induction on $\delta(r) \le \delta(f),\delta(g)$ over the common left-factors of $f$ and $g$ produces a greatest common left-factor. The other direction is analogous.
\end{proof}

We say that $\cat$ is right/left-Ore if it is cancellative and has common right/left-multiples.

\begin{theorem}
A category $\cat$ that is right-Ore embeds in a groupoid $\grpd$ such that every element $h \in \grpd$ can be written as $h = fg^{-1}$ with $f, g \in \cat$.
\end{theorem}

The groupoid $\grpd$ in the theorem is called the \emph{Ore localization} 
$\Ore(\cat)$ of $\cat$. Using the universal property, it is not hard to see that 
it coincides with the enveloping groupoid of $\cat$.

The fundamental group of an Ore category has a particularly easy description. In general, an element of $\pi_1(\cat,x)$ is represented by a sequence $f_0g_1^{-1}f_1 \cdots f_{n-1}g_n^{-1}$ with $f_i, g_i \in \cat(x_i,-)$ and $f_j, g_{j+1} \in \cat(-,y_j)$. But if $\cat$ has common right-multiples, then $g_1^{-1}f_1$ can be rewritten as $f_1'{g_1'}^{-1}$ and so the sequence can be shortened to $(f_0f_1')(g_2g_1')^{-1}f_2 \cdots f_{n-1}g_n^{-1}$. Iterating this argument, we find that every element of $\pi_1(\cat,x)$ is of the form $fg^{-1}$ with $f,g \in \cat(x,-)$.

\subsection{Presentations}

We introduce presentations for categories. This is analogous to the situation for monoids and we will be brief. See \cite[Section~II.1.4]{dehornoy15} for details.

A (small) \emph{precategory} $\precat$ consists of a set of objects $\Ob(\precat)$ and a set of morphisms $\precat$. As for categories, each morphism has a \emph{source} and a \emph{target} that are objects and it is a morphism from the source to its target. The set of morphisms from $y$ to $x$ is denoted $\precat(y,x)$. The monoidal aspects of a category are missing in a precategory: it does not have identities or a composition.

Given a precategory $\precat$ there exists a free category $\precat^*$ generated by $\precat$. It has the universal property that if $\phi \colon \precat \to \cat$ is a morphism of precategories and $\cat$ is a category, then $\phi$ uniquely factors through $\precat \to \precat^*$. One can construct $\precat^*$ to have the same objects as $\precat$ and have morphisms finite words in $\precat$ that are composable.

A \emph{relation} is a pair $r = s$ of morphisms in $\precat^*$ with same source and target. If $\phi \colon \precat^* \to \cat$ is a morphism, the relation \emph{holds} in $\cat$ if $\phi(r) = \phi(s)$. A \emph{presentation} consists of a precategory $\precat$ and a family of relations $\catrels$ in $\precat^*$. The category it presents is denoted $\gen{\precat \mid \catrels}$.

It has the universal property that if $\phi \colon \precat \to \cat$ is a morphism of precategories and $\cat$ is a category in which all relations in $\catrels$ hold then $\phi$ uniquely factors through $\precat \to \gen{\precat \mid \catrels}$. One can construct $\gen{\precat \mid \catrels}$ by quotienting $\precat^*$ by the symmetric, transitive closure of the relations.

\subsection{Garside families}

The following notions are at the core of \cite{dehornoy15}. We will sometimes be needing the notions with the reverse order. What in \cite{dehornoy15} referred to as a Garside family in a left-cancellative category, will be called a left-Garside family here to avoid confusion in categories that are left- and right-cancellative.

Let $\cat$ be a left-cancellative category and let $\gars \subseteq \cat$ be a set of morphisms. We denote by $\gars^\sharp$ the set $\cat^\times \cup \gars\cat^\times$ of morphisms that are invertible or left-multiples of invertibles by elements of $\gars$. We say that $\gars^\sharp$ is \emph{closed under (left/right-) factors} if every (left/right-) factor of an element in $\gars^\sharp$ is again in $\gars^\sharp$. An element $s \in \gars$ is an \emph{$\gars$-head} of $f \in \cat$ if $s$ is a left-factor of $f$ and every left-factor of $f$ is a left-factor of $s$ \cite[Definition~IV.1.10]{dehornoy15}. The set $\gars$ is a \emph{left-Garside family} if $\gars^\sharp$ generates $\cat$, is closed under right-factors, and every non-invertible element of $\cat$ admits and $\gars$-head \cite[Proposition~IV.1.24]{dehornoy15}. If $\gars$ is a left-Garside family then $\cat^\times \gars \subseteq \gars^\sharp$, so in fact $\gars^\sharp = \cat^\times \cup \cat^\times \gars\cat^\times$ \cite[Proposition~III.1.39]{dehornoy15}.

All notions readily translate to right-Garside families, except that the head is called an \emph{$\gars$-tail} if $\gars$ is a right-Garside family. Note that $\gars^\sharp$ is defined as $\cat^\times \cup \cat^\times\gars$ when $\gars$ is (regarded as) a right-Garside family.

We will be interested in Garside families that are closed under factors. We describe two situations where this is the case.

Let $\cat$ be left-cancellative and consider a map $\Delta \colon \Ob(\cat) \to \cat$ with $\Delta(x) \in \cat(x,-)$. We write 
\begin{align*}
\Div(\Delta) &= \{g \in \cat\mid \cat(x,-) \ni g\ \exists h \in \cat\ gh = \Delta(x)\}\\
\tDiv(\Delta) &= \{h \in \cat\mid \exists x \exists g \in \cat(x,-)\ gh = \Delta(x)\}
\end{align*}
for the families of left- respectively right-factors of morphisms in the image of $\Delta$. Such a map is a \emph{right-Garside map} if $\Div(\Delta)$ generates $\cat$, if $\tDiv(\Delta) \subseteq \Div(\Delta)$, and if for every $g \in \cat(x,-)$ the elements $g$ and $\Delta(x)$ admit a greatest common left-factor.
If $\Delta$ is a right-Garside map then $\Div(\Delta)$ is a left-Garside family closed under left-factors and thus under factors \cite[Proposition~V.1.20]{dehornoy15}. We note the following for future reference.

\begin{observation}
\label{obs:gars_map_to_family}
Let $\cat$ be a left-cancellative, factor-finite category and let $\Delta$ be a right-Garside map. Then $\gars \defeq \Div(\Delta)$ is a left-Garside family closed under factors and $\gars(x,-)$ is finite for every $x \in \Ob(\cat)$.
\end{observation}

Let $\cat$ be right-Ore. A right-Garside family is \emph{strong} if for $s,t \in \gars^\sharp$ there exist $s',t' \in \gars^\sharp$ such that $st' = ts'$ is a least common right-multiple of $s$ and $t$ \cite[Definition~2.29]{dehornoy15}. If $\gars$ is a strong right-Garside family then $\gars^\sharp$ is also closed under left-factors and thus is closed under factors \cite[Proposition~1.35]{dehornoy15}.


\section{Fundamental examples}
\label{sec:f}
\subsection[Thompson's group $F$ and the category $\thomcat$]{Thompson's group $\mathbf{F}$ and the category $\mathbf{\thomcat}$}

Our description of Thompson's groups is not the standard one, which can be found in \cite{cannon96}. An element of Thompson's group $F$ is given by a pair $(T_+,T_-)$ of finite rooted binary trees with the same number of leaves, say $n$. If we add a caret to the $i$th leaf ($1 \le i \le n$) of $T_+$, that is we make it into an inner vertex with to leaves below it, we obtain a tree $T_+'$ on $n+1$ vertices. If we also add a caret to the $i$th leaf of $T_-$ we obtain another tree $T_-'$. We want to regard $(T_+',T_-')$ as equivalent to $(T_+,T_-)$ so we take the reflexive, symmetric, transitive closure of the operation just described and write the equivalence class by $[T_+,T_-]$. Thompson's group $F$ is the set of equivalence classes $[T_+,T_-]$.

In order to define the product of two elements $[T_+,T_-]$ and $[S_+,S_-]$ we note that we can add carets to both tree pairs to get representatives $[T_+',T'] = [T_+,T_-]$ and $[T',T_-'] = [S_+,S_-]$ where the second tree of the first element and the first tree of the second element are the same. Therefore multiplication is completely defined by declaring that $[T_+',T'] \cdot [T',T_-'] = [T_+',T_-']$. It is easy to see that $[T,T]$ is the neutral element for any tree $T$ and that $[T_+,T_-]^{-1} = [T_-,T_+]$.

We have defined the group $F$ in such a way that a categorical description imposes itself, cf.\ \cite{belk04}. We define $\thomcat$ to be the category whose objects are positive natural numbers and whose morphisms $m \from n$ are binary forests on $m$ roots with $n$ leaves. Multiplication of a forest $E \in \thomcat(\ell,m)$ and a forest $F \in \thomcat(m,n)$ is defined by identifying the leaves of $E$ with the roots of $F$ and taking $EF$ to be the resulting tree. Pictorially this corresponds to stacking the two forests on top of each other (see Figure~\ref{fig:forest_mult}).

\begin{figure}[t]
\centering
\begin{tikzpicture}[line width=0.8pt, scale=0.44]
  \draw
   (-3,-3) -- (0,0) -- (3,-3)   (1,-3) -- (-1,-1)   (-1,-3) -- (0,-2);
  \filldraw
   (-3,-3) circle (1.5pt)   (0,0) circle (1.5pt)   (3,-3) circle (1.5pt)   (-1,-1) circle (1.5pt)   (-1,-3) circle (1.5pt)   (0,-2) circle (1.5pt)   (1,-3) circle (1.5pt)   (2,0) circle (1.5pt)   (4,0) circle (1.5pt)   (6,0) circle (1.5pt);
  \node at (8,0) {$\dots$};

  \begin{scope}[yshift=-5cm]
   \draw
    (-4,-1) -- (-3,0) -- (-2,-1)   (1,-2) -- (3,0) -- (5,-2)   (3,-2) -- (2,-1);
   \filldraw
    (-4,-1) circle (1.5pt)   (-3,0) circle (1.5pt)   (-2,-1) circle (1.5pt)   (-1,0) circle (1.5pt)   (1,0) circle (1.5pt)   (1,-2) circle (1.5pt)   (3,0) circle (1.5pt)   (5,-2) circle (1.5pt)   (3,-2) circle (1.5pt)   (2,-1) circle (1.5pt)   (5,0) circle (1.5pt);
  \node at (7,0) {$\dots$};   \node at (10,2) {$=$};
  \end{scope}
   
  \begin{scope}[xshift=16cm]
   \draw
    (-3,-3) -- (0,0) -- (3,-3)   (1,-3) -- (-1,-1)   (-1,-3) -- (0,-2)
    (-4,-4) -- (-3,-3) -- (-2,-4)   (1,-5) -- (3,-3) -- (5,-5)   (3,-5) -- (2,-4);
   \filldraw
   (-3,-3) circle (1.5pt)   (0,0) circle (1.5pt)   (3,-3) circle (1.5pt)   (-1,-1) circle (1.5pt)   (-1,-3) circle (1.5pt)   (0,-2) circle (1.5pt)   (1,-3) circle (1.5pt)   (2,0) circle (1.5pt)   (4,0) circle (1.5pt)   (6,0) circle (1.5pt)   (-4,-4) circle (1.5pt)   (-2,-4) circle (1.5pt)   (1,-5) circle (1.5pt)   (5,-5) circle (1.5pt)   (3,-5) circle (1.5pt);
   \node at (8,0) {$\dots$};
  \end{scope}
\end{tikzpicture}
\caption{Multiplication of forests (taken from \cite{witzel}).}
\label{fig:forest_mult}

\end{figure}
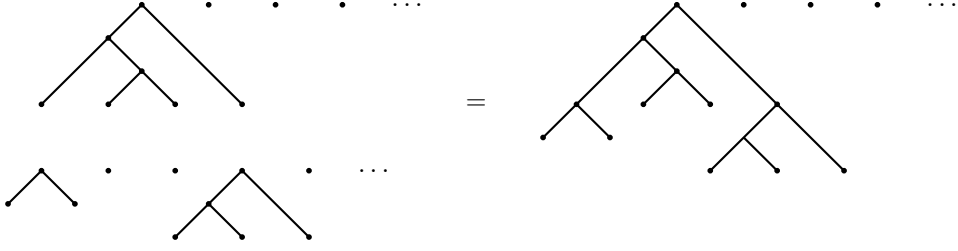
\begin{proposition}
\label{prop:thomcat_ore}
The category $\thomcat$ is strongly Noetherian and right-Ore. In fact, it has least common right-multiples and greatest common left-factors.
\end{proposition}

\begin{proof}
The identity map $\rho \colon \N = \Ob(\thomcat) \to \N$ is a height function on $\thomcat$. Thus $\thomcat$ is strongly Noetherian.

The least common right-multiple of two forests in $\thomcat(m,-)$ is their union (regarding both forests as subforests of the leafless binary forest on $m$ roots). The greatest common left-factor is their intersection. Left cancellativity means that given a forest $f \in \thomcat(m,\ell)$ and a left-factor $a \in \thomcat(m,n)$ the forest in $b \in\thomcat(n,\ell)$ with $f = ab$ is unique. Indeed it is the forest obtained from $f$ by removing $a$ and turning the leaves of $a$ into roots. Right cancellativity means that $a$ is uniquely determined if $f = ab$. To see this, we identify the leaves of $f$ with the leaves of $b$. Now the common predecessor in $f$ of a set of leaves of a tree of $b$ is a leaf of $a$ and every leaf of $a$ arises in that way.
\end{proof}

The proposition together with the remark at the end of Section~\ref{sec:ore_property} shows that every element of $\pi_1(\thomcat,1)$ is represented by $fg^{-1}$ where $f,g \in \thomcat(1,-)$ are binary trees. Cancellativity ensures that $fg^{-1} = {f'}{g'}^{-1}$ if and only if there exist $h$ and $h'$ such that $fh=f'h'$ and $gh=g'h'$. Comparing this description with our definition of $F$ we see:

\begin{proposition}
Thompson's group $F$ is isomorphic to $\pi_1(\thomcat,1)$.
\end{proposition}

Later on it will be convenient to have a presentation for $\thomcat$. The shape of the relations will not come as a surprise to the reader familiar with Thompson's groups. A proof can be found in \cite{witzel16b}.

\begin{proposition}
\label{prop:thomcat_presentation}
The category $\thomcat$ has a presentation with generators morphisms $\lambda_i^n \colon n \from n+1$ for $1 \le i \le n$ subject to the relations
\begin{equation}
\label{eq:caret_relation}
\lambda^n_i\lambda^{n+1}_j = \lambda^n_j\lambda^{n+1}_{i+1} \quad \text{for}\quad 1 \le j < i \le n\text{.}
\end{equation}
Every morphism in $\thomcat(m,n)$ can be written in a unique way as $\lambda_{i_m}^m \cdots \lambda_{i_{n-1}}^{n-1}$ with $(i_j)_j$ non-decreasing.
\end{proposition}

\begin{remark}\label{rem:commuting}
The relations \eqref{eq:caret_relation} reflect a commutation phenomenon: for any forest adding a caret to the $i$th leaf and then to the $j$th leaf has the same effect as doing it the other way around. That it does not algebraically look like a commutation relation is due to the fact the the index of the right one of the two leaves has changed when adding the left caret. This is inevitable in the present setup because the $i$th leaf has no identity as a particular vertex in the infinite rooted binary tree but simultaneously represents all $i$th leaves of trees with $n$ leaves. A larger category in which the relations are algebraically commutation relations will appear in Section~\ref{sec:graph_rewriting}.
\end{remark}

Note that since $\thomcat$ is connected, the fundamental groups at different objects are isomorphic. This corresponds to the elementary fact that the tree pair $(T_+,T_-)$ representing an element of $F$ can always be chosen so that $T_+,T_-$ contain an arbitrary fixed subtree.

The most convenient way to exhibit a Garside family in $\thomcat$ is by describing a right-Garside map: for every $n \in \N = \Ob(\thomcat)$ let $\Delta(n)$ be the forest where every tree is a single caret.

\begin{proposition}
\label{prop:f_garside_map}
The map $\Delta \colon \Ob(\thomcat) \to \thomcat$ is a right-Garside map.
\end{proposition}

\begin{proof}
The family $\Div(\Delta)$ consists of morphisms where every forest is either a single caret or trivial. Every forest can be built of from these, for example by adding one caret at a time. This shows that $\Div(\Delta)$ generates $\thomcat$. The family $\tDiv(\Delta)$ also consists of morphisms where every forest is either a single caret or trivial with the additional condition that the total number of leaves is even and the left leaf of every caret has an odd index. In particular $\tDiv(\Delta) \subseteq \Div(\Delta)$. If $g \in \thomcat(x,-)$ then $g$ and $\Delta(x)$ have a greatest common left-factor by Proposition~\ref{prop:thomcat_ore}.
\end{proof}

With Observation~\ref{obs:gars_map_to_family} we get:

\begin{corollary}
\label{cor:f_garside_family}
The category $\thomcat$ admits a left-Garside family $\gars$ that is closed under factors such that $\gars(x,-)$ is finite for every $x \in \thomcat$.
\end{corollary}

\begin{remark}
The family $\Div(\Delta)$ is in fact a right- as well as a left-Garside family. It is strong as a right-Garside family but not as a left-Garside family.
\end{remark}

If instead of rooted binary trees one takes rooted $n$-ary trees ($n \ge 2$) in the description above, one obtains the category $\thomcat_n$. Everything is analogous to $\thomcat$ but the new aspect that occurs for $n > 2$ is that the category is no longer connected: the number of leaves of an $n$-ary tree with $r$ roots will necessarily be congruent to $r$ modulo $n-1$, hence there is no morphism in $\thomcat_n$ connecting objects that are not congruent modulo $n-1$. As a consequence, the point at which the fundamental group is taken does matter and we obtain $n-1$ different groups for each category. It turns out, however, that the fundamental groups are in fact isomorphic independently of the basepoint and so one only defines
\[
F_{n} = \pi_1(\thomcat_n,1)\text{.}
\]
The groups $F_n$ are the smallest examples of the \emph{Higman--Thompson groups} introduced by Higman~\cite{higman74}. As we will see later, the fundamental groups of the different components are non-isomorphic in the categories for the larger Higman--Thompson groups.

\subsection{Braid groups}
\label{sec:garside}

The \emph{braid group} on $n$ strands, introduced by Artin~\cite{artin25}, is the group given by the presentation
\begin{align}
\Braid_n =
\left\langle \sigma_1,\ldots,\sigma_{n-1} \begin{array}{c|cl} 
		&\sigma_i\sigma_j = \sigma_j\sigma_i, &\abs{i-j} \ge 2,\\
&\sigma_i\sigma_{i+1}\sigma_i = \sigma_{i+1}\sigma_i\sigma_{i+1}, &1 \le i \le n-2
	\end{array}\right\rangle\text{.}\label{eq:braid}
\end{align}
Its elements, called \emph{braids}, can be conveniently depicted as braid diagrams as in Figure~\ref{fig:braid_diagram}, illustrating a physical interpretation as braids on $n$ strands. The first relations are \emph{commutation relations}, the second are \emph{braid relations}. The group $\Braid_n$ arise as the fundamental group of the configuration space of $n$ unordered points in the disc and as the mapping class group of the $n$-punctured disc, see \cite{birman74,kassel08} for more details.

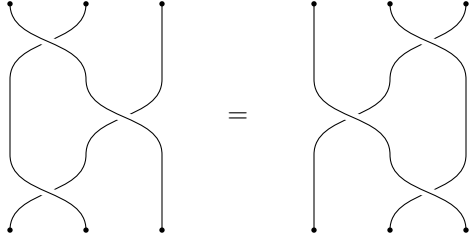
\begin{figure}[htb]
\centering
\begin{tikzpicture}[scale=1,yscale=-1]
\begin{scope}
\draw
   (2,0) -- (2,1) to [out=90, in=-90, looseness=1] (1,2)  to [out=90, in=-90, looseness=1] (0,3);
\draw[white, line width=4pt]
   (1,0) to [out=90, in=-90, looseness=1] (0,1) -- (0,2) to [out=90, in=-90, looseness=1] (1,3);
\draw
   (1,0) to [out=90, in=-90, looseness=1] (0,1) -- (0,2) to [out=90, in=-90, looseness=1] (1,3);
\draw[white, line width=4pt]
   (0,0) to [out=90, in=-90, looseness=1] (1,1) to [out=90, in=-90, looseness=1] (2,2) -- (2,3);
\draw
   (0,0) to [out=90, in=-90, looseness=1] (1,1) to [out=90, in=-90, looseness=1] (2,2) -- (2,3);
\foreach \x in {0,1,2}
\node[end] at (\x,0) {};
\foreach \x in {0,1,2}
\node[end] at (\x,3) {};
\end{scope}
\node at (3,1.5) {$=$};
\begin{scope}[shift={(4,0)}]
\draw
   (2,0) to [out=90, in=-90, looseness=1] (1,1)  to [out=90, in=-90, looseness=1] (0,2) -- (0,3);
\draw[white, line width=4pt]
   (1,0) to [out=90, in=-90, looseness=1] (2,1) -- (2,2) to [out=90, in=-90, looseness=1] (1,3);
\draw
   (1,0) to [out=90, in=-90, looseness=1] (2,1) -- (2,2) to [out=90, in=-90, looseness=1] (1,3);
\draw[white, line width=4pt]
   (0,0) -- (0,1) to [out=90, in=-90, looseness=1] (1,2) to [out=90, in=-90, looseness=1] (2,3);
\draw
   (0,0) -- (0,1) to [out=90, in=-90, looseness=1] (1,2) to [out=90, in=-90, looseness=1] (2,3);
\foreach \x in {0,1,2}
\node[end] at (\x,0) {};
\foreach \x in {0,1,2}
\node[end] at (\x,3) {};
\end{scope}
\end{tikzpicture}
\caption{Diagrams illustrating the braid relation.}
\label{fig:braid_diagram}
\end{figure}

What is known as Garside Theory today arose out of Garside's study of braid groups \cite{garside69}. In this classical case, the category $\cat$ has a single object and thus is a monoid. Specifically, a \emph{Garside monoid} is a monoid $M$ with an element $\Delta \in M$, called a \emph{Garside element}, such that
\begin{enumerate}
\item $M$ is cancellative and has least common right- and left-multiples and greatest common right- and left-factors,
\item the left- and right-factors of $\Delta$ coincide, they are finite in number, and generate $M$,
\item there is a map $\delta \colon M \to \N$ such that $\delta(fg) \ge \delta(f) + \delta(g)$ and $\delta(g) > 0$ if $g \ne 1$.\label{item:garside_length}
\end{enumerate}
A \emph{Garside group} is the group of fractions of a Garside monoid. Among the main features of Garside groups is that they have solvable word problem and conjugacy problem.

Note that a Garside monoid, regarded as a category with one object is, by definition, left- and right-Ore and strongly Noetherian. Moreover the family of factors of $\Delta$ is a left- and right-Garside family.

To see that braid groups are in fact Garside groups consider the \emph{braid monoid} $\Braid_n^+$. It is obtained by interpreting the presentation \eqref{eq:braid} as a monoid presentation. It is a non-trivial consequence of Garside's work that the obvious map $\Braid_n^+ \to \Braid_n$ is injective so that the braid monoid can be regarded as a subset of the braid groups. Its elements are called \emph{positive braids} and are characterized by the property that left strands always overcrosses the right strand. The element $\Delta$ in $\Braid_n^+$ is the braid that performs a full half twist and is characterized by the fact that every strand crosses every other strand precisely once, see Figure~\ref{fig:delta}. Its (left- or right-) factors are the braids where every strand crosses every other strand at most once. The function $\delta$ is simply number of crossings, which is the same as length as a word in the generators. Now $\Braid_n^+$ is a Garside monoid with Garside element $\Delta$, see \cite[Section~I.1.2, Proposition~IX.1.29]{dehornoy15}. Its group of fractions is $\Braid_n$ which is therefore a Garside group.

\begin{figure}
\reflectbox{\includegraphics{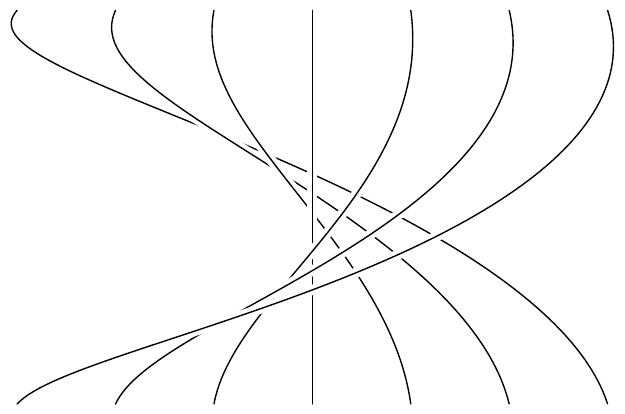}}
\caption{The element $\Delta$ in $\Braid_7^+$.}
\label{fig:delta}
\end{figure}

It was noted by Birman--Ko--Lee \cite{birman98} that there is in fact another monoid $\Braid_n^{*+}$, called the \emph{dual braid monoid}, that also admits a Garside element $\Delta^*$ and has $\Braid_n$ as its group of fractions, see also \cite[Section~I.1.3]{dehornoy15}. This monoid is in many ways better behaved than $\Braid_n^+$. Brady \cite{brady01} used the dual braid monoid to construct a finite classifying space for the braid group.

Note that adding the relations $\sigma_i^2$ to the presentation \eqref{eq:braid} results in a presentation for the symmetric group $\Sym_n$. In particular, there is a surjective homomorphism $\pi \colon \Braid_n \to \Sym_n$ that takes $\sigma_i$ to the transposition $s_i \defeq (i\ i+1)$.

The symmetric group is a finite Coxeter group and the braid group is its corresponding \emph{Artin group}. For every Coxeter system $(W,S)$ there exists an Artin group $A_W$ obtained analogously and a morphism $\pi \colon A_W \to W$. Whenever $W$ is finite, the Artin group $A_W$ again contains a Garside monoid and a dual Garside monoid, see \cite{brieskorn72,bessis03}.


\section{Finiteness properties of fundamental groups of Ore categories}
\label{sec:finiteness_properties}

A classifying space for a group $G$ is a CW complex $B$ whose fundamental group is $G$ and whose universal cover $X = \tilde{B}$ is contractible. Since $G$ acts freely on $X$ with quotient $B = G \backslash X$, one can equivalently say that a classifying space is the quotient of a contractible CW complex by a free $G$-action. Our goal in this section is to construct ``good'' classifying spaces for fundamental groups of Ore categories. The best classifying spaces are compact ones; they have finitely many cells so we also refer to them as \emph{finite}. If $G$ admits a finite clasifying space, we say that it is of type $F$. If a finite classifying space does not exist we aim at classifying spaces with weaker finiteness properties. We start by constructing an action on a contractible space.


\subsection{Contractible spaces from Ore categories with Garside families}
\label{sec:space_ore_garside}

Let $\cat$ be a category that is right-Ore and strongly Noetherian. Let $\gars$ be a left- or right-Garside family such that $\gars^\sharp$ is closed under factors. Let $* \in \Ob(\cat)$ be a base object. Our goal is to construct a contractible space $X$ on which $\pi_1(\cat,*)$ acts with good finiteness properties of the stabilizers as well as the quotient. In the whole discussion $\cat$ can be replaced by the component of $*$ in $\cat$ so all assumptions only need to be made for objects and morphisms in that component.

We put $\eltry = \gars^\sharp$ and recall that $\eltry = \cat^\times \cup \cat^\times \gars \cat^\times$. We call the elements of $\eltry$ \emph{elementary}. Let $\delta \colon \cat \to \N$ be a map that witnesses strong Noetherianity. Note that if $f \in \cat(x,y)$ and $g \in \cat^\times(-,x)$ and $h \in \cat^\times(y,-)$ are invertible then
\[
-\delta(g^{-1}) + \delta(f) - \delta(h^{-1}) \ge \delta(gfh) \ge \delta(g) + \delta(f) + \delta (h)
\]
so $\delta(f) =\delta(gfh)$ and $\delta$ is invariant under pre- and postcomposition by invertibles.

We define the set $P = \Ore(\cat)(*,-)/\cat^\times$, that is, elements of $P$ are equivalence classes $\bar{a}$ of elements $a \in \Ore(\cat)(*,-)$ modulo the equivalence relation that $\bar{a} = \bar{a\smash{'}}$ if there exists a $g \in \cat^\times$ with $ag = a'$. We define a relation $\le$ on $P$ by declaring $\bar{a} \le \bar{b}$ if there exists an $f \in \cat$ with $af = b$.

\begin{lemma}
\label{lem:poset_contractible}
The relation $\le$ is a partial order on $P$ in which any two elements have a common upper bound. In particular, the realization $\abs{P}$ is contractible.
\end{lemma}

\begin{proof}
Note first that whether $f = a^{-1}b$ lies in $\cat$ is independent of the representatives. Reflexivity and transitivity are clear. If $\bar{a} \le \bar{b} \le \bar{a}$ then there exist $f, h \in \cat$ and $g \in \cat^\times$ such that $af=b$ and $bh=ag$ showing that $fh$ is a unit. In particular, $f$ has a right-inverse and $h$ has a left-inverse so $f$ and $h$ are units by Lemma~\ref{lem:cancellative_inverse}. This shows $\bar{a} = \bar{b}$.

For any $a \in \Ore(\cat)$ there is an $f \in \cat$ such that $af \in \cat$. Since $\cat$ has common right-multiples, it follows that for any two elements $a_1,a_2 \in \Ore(\cat)$ there exist $f_1, f_2 \in \cat$ with $a_1f_1 = a_2f_2$.
\end{proof}

We define a second, more restrictive relation $\preceq$ on $P$ by declaring that $\bar{a} \preceq \bar{b}$ if there exists an $e \in \eltry$ with $ae = b$. Note that this relation will typically not be transitive. However, if $\bar{a} \preceq \bar{b}$ and $\bar{a} \le \bar{c} \le \bar{b}$ then $\bar{a} \preceq \bar{c} \preceq \bar{b}$ because $\eltry$ is closed under factors. The complex $X \subseteq \abs{P}$ consists of those chains in $\abs{P}$ that are chains with respect to $\preceq$. In particular, $P$ is the vertex set of $X$.

\begin{proposition}
\label{prop:stein_farley_contractible}
The complex $X$ is contractible.
\end{proposition}

\begin{proof}
Note that $X$ is a subspace of $\abs{P}$ containing all the vertices. One can obtain $\abs{P}$ from $X$ by gluing in (realizations of) intervals $[\bar{a},\bar{b}]$ not yet contained in $X$. To organize the gluing, note the following: if $[\bar{c}, \bar{d}]$ is a proper subinterval of $[\bar{a},\bar{b}]$ with $f = a^{-1}b \in \cat$ and $h = c^{-1}d \in \cat$ then $h$ is a proper factor of $f$. To an interval $[\bar{a},\bar{b}]$ with $f = a^{-1}b$ we assign the height $\hat{\delta}([\bar{a},\bar{b}]) = \delta(f)$. Note that this is well-defined, because any other representative $f'$ will differ from $f$ only by invertibles and $\delta$ is invariant under pre- and postcomposition by invertibles. Note also that proper subintervals have strictly smaller $\hat{\delta}$-value. We can therefore glue in the intervals with increasing value of $\hat{\delta}$ and are sure that when we glue in an interval, any proper subinterval is already glued in.

For any $n \in \N$ let $\abs{P}_{\hat{\delta} < n}$ be the subcomplex of $\abs{P}$ consisting of $X$ and intervals of $\hat{\delta}$-value $< n$. If $X$ was not contractible, there would be a sphere in $X$ that could not be contracted in $X$ but in $\abs{P}$. The contraction would be compactly supported, hence use simplices supported on finitely many simplices. It therefore suffices to show that the inclusion $X \to \abs{P}_{\hat{\delta} < n}$ is a homotopy equivalence for all $n \in \N$.

For $n = 0$ this is clear, so assume $n > 0$. Then
\[
\abs{P}_{\hat{\delta} < n} = \abs{P}_{\hat{\delta} < n-1} \cup \bigcup_{\hat{\delta}([\bar{a},\bar{b}]) = n-1} \abs{[\bar{a},\bar{b}]}\text{.}
\]
The intervals that are glued in meet only in $\abs{P}_{\hat{\delta} < n-1}$ and they are glued in along $\abs{[\bar{a},\bar{b})} \cup \abs{(\bar{a},\bar{b}]}$. This is a suspension of $\abs{(\bar{a},\bar{b})}$ and so it suffices to show that the open interval is contractible.

If $\gars$ is a left-Garside family, every element $h$ of $\cat$, and every left-factor of $f$ in particular, has an $\gars$-head $\head(g)$. We define the map $\theta \colon [\bar{a},\bar{b}] \to [\bar{a},\bar{b}]$ by $\overline{ah} \mapsto \overline{a\head(h)}$. Note that $\theta(\bar{b}) < \bar{b}$ because otherwise $[\bar{a},\bar{b}]$ is already contained in $\abs{P}$. Note also that $\theta(\bar{c}) > \bar{a}$ for $\bar{c} > \bar{a}$ because the head of a non-invertible is not invertible. This shows that $\theta$ restricts to a map $(\bar{a},\bar{b}) \to (\bar{a},\bar{b})$ with $\bar{c} \ge \theta(\bar{c}) \le \theta(\bar{b})$ and we can apply \cite[Section~1.5]{quillen78} to see that $\abs{(\bar{a},\bar{b})}$ is contractible.

If $\gars$ is a right-Garside family, $\theta$ is defined by $\overline{bh^{-1}} \mapsto \overline{b\tail(h)^{-1}}$. For the same resasons as above $\theta$ restricts to a map $(\bar{a},\bar{b}) \to (\bar{a},\bar{b})$ with $\bar{c} \le \theta(\bar{c}) \ge \theta(\bar{a})$ and we can again apply \cite[Section~1.5]{quillen78}.
\end{proof}

There is an obvious action of $\pi_1(\cat,*)$ on $X$ which is given by precomposition: if $g \in \pi_1(\cat,*) = \Ore(\cat)(*,*)$ and $a \in \Ore(\cat)(*,-)$ then $g \bar{a} = \overline{ga}$ and the relations $\le$ and $\preceq$ are clearly preserved under this action.

Next we want to look at stabilizers and weak fundamental domains. These will be particularly well-behaved with an additional assumption. We say that $\gars$ is \emph{(right-)locally finite} if for every object $x \in \Ob(\cat)$ the set $\gars(x,-)$ is finite up to pre- and post-composition by invertibles. Local finiteness of $\gars$ does \emph{not} imply that $X$ is locally finite but does imply:

\begin{observation}
\label{obs:loc_fin}
Assume that $\gars$ is locally finite.
For every $\bar{a} \in P$ there are only finitely many $\bar{b} \in P$ with $\bar{a} \preceq \bar{b}$. In particular, there are only finitely many simplices for which $\bar{a}$ is $\preceq$-minimal.
\end{observation}

\begin{lemma}
\label{lem:stabilizers}
Every simplex-stabilizer of the action of $\pi_1(\cat,*)$ on $X$ is isomorphic to a subgroup of $\cat^\times(x,x)$ for some $x \in \Ob(\cat)$. If $\gars$ is locally finite, the subgroup has finite index.
\end{lemma}

\begin{proof}
Let $\bar{a}$ be a vertex in $X$ with $a \in \Ore(\cat)(*,x)$ and suppose that $g \in \pi_1(\cat,*)$ fixes $\overline{a}$, that is, $\overline{a} = g\overline{a} = \overline{ga}$. Then $a^{-1}ga \in \cat^\times(x,x)$. This shows that the stabilizer of $\bar{a}$ is conjugate to $\cat^\times(x,x)$. If $\gars$ is locally finite then Observation~\ref{obs:loc_fin} implies that the stabilizer of an arbitrary simplex has finite index in a vertex stabilizer.
\end{proof}

\begin{corollary}
\label{cor:free_proper}
If $\cat^\times(x,x) = \{1_x\}$ for every object $x \in \Ob(\cat)$ then the action of $\pi_1(\cat,*)$ on $X$ is free. If $\cat^\times(x,x)$ is finite then the action is proper.
\end{corollary}

Now let us pick, for every $x \in \Ob(\cat)$ a morphism $f_x \in \Ore(\cat)(*,x)$ arbitrarily and let $K_x \subseteq X$ be the union of the realizations of the intervals $[\overline{f_x},\overline{f_xe}]$ with $e \in \eltry(x,-)$.

\begin{lemma}
\label{lem:weak_fundamental_domain}
The complex $X$ is covered by the $\pi_1(\cat,*)$-translates of the complexes $K_x, x \in \Ob(\cat)$. If $\gars$ is locally finite then each $K_x$ is compact.
\end{lemma}

\begin{proof}
If $\sigma = \{f \prec fe_1 \prec \ldots \prec fe_k\}$ is a simplex in $X$ with $f \in \Ore(\cat)(*,x)$ and $e_1, \ldots,e_k \in\eltry(x,-)$ then $f_xf^{-1} \in \pi_1(\cat,*)$ and $f_xf^{-1} K_x$ contains $\sigma$. The second statement is clear.
\end{proof}

The ideal special case is:

\begin{corollary}
If $\cat$ has no non-identity invertible morphisms and has only finitely many objects and if $\gars$ is locally finite then $\pi_1(\cat,*)$ has a finite classifying space.
\end{corollary}

\begin{proof}
Under the assumption the action of $\pi_1(\cat,*)$ is free by Lemma~\ref{cor:free_proper} and cocompact by Lemma~\ref{lem:weak_fundamental_domain}. The quotient is then finite a classifying space.
\end{proof}

In particular we recover the main result of \cite{charney04}.

\begin{corollary}
\label{cor:garside_complex}
Every Garside group $G$ has a finite classifying space.
\end{corollary}

In the case of the dual braid monoid, the complex we constructed is precisely the \emph{dual Garside complex} constructed by Brady~\cite{brady01}.


\subsection{Finiteness properties}

Topological finiteness properties of a group $G$ were introduced by Wall \cite{wall65,wall66} and are conditions on how finite a classifying space for $G$ can be chosen. A group is said to be \emph{of type $F_n$} if it admits a classifying space $B$ whose $n$-skeleton $B^{(n)}$ has finitely many cells. Equivalently a group is of type $F_n$ if it acts freely on a contractible space $X$ such that the action on $X^{(n)}$ is cocompact. It is clear that type $F_n$ implies type $F_m$ for $m < n$ and one defines the finiteness length $\phi(G)$ to be the supremal $n$ for which $G$ is of type $F_n$. If $\phi(G) = \infty$ then $G$ is said to be of type $F_\infty$.

In low dimensions, these properties have familiar descriptions: a group is of type $F_1$ if and only if it is finitely generated, and it is of type $F_2$ if and only if it is finitely presented.

Given a group $G$, in order to study its finiteness properties, one needs to let $G$ act on a highly connected space $X$. If the action is free, then the low-dimensional skeleta of $G \backslash X$ are those of a classifying space. A useful result is Brown's criterion which says that one does not have to look at free actions, see \cite[Propositions~1.1,~3.1]{brown87}:

\begin{theorem}
\label{thm:browns_criterion}
Let $G$ act cocompactly on an $(n-1)$-connected CW complex $X$. If the stabilizer of every $p$-cell of $X$ is of type $F_{n-p}$ then $G$ is of type $F_n$.
\end{theorem}

The full version of Brown's criterion also gives a way to decide that a group is not of type $F_n$. We formulate it here only to explain why we will not be able to apply it:

\begin{theorem}
\label{thm:browns_criterion_negative}
Let $G$ act on an $(n-1)$-connected CW complex $X$ and assume that the stabilizer of every $p$-cell of $X$ is of type $F_{n-p}$. If $G$ is of type $F_n$ then for every cocompact subspace $Y$ and any basepoint $* \in Y$ there exists a cocompact subspace $Z \supseteq Y$ such that the maps $\pi_k(Y,*) \to \pi_k(Z,*)$ induced by inclusion have trivial image for $k \le n-1$.
\end{theorem}

Theorem~\ref{thm:browns_criterion_negative} can be used to show that a group is not of type $F_n$ if this is visible in the topology of $X$. On the other hand, if the stabilizers have bad finiteness properties we cannot decide whether $G$ has good finiteness properties or not: in that case we are looking at the wrong action.


\subsection{Combinatorial Morse theory}

In order to study connectivity properties of spaces and apply Brown's criterion we will be using combinatorial Morse theory as introduced by Bestvina and Brady \cite{bestvina97}. Here we give the most basic version used in Section~\ref{sec:proof_scheme}.

Let $X$ be the realization of an abstract simplicial complex, regarded as a CW complex. A \emph{Morse function} is a function $\rho \colon X^{(0)} \to \N$ with the property that $\rho(v) \ne \rho(w)$ if $v$ is adjacent to $w$. For $n \in \N$ the sublevel set $X_{\rho < n}$ is defined to be the full subcomplex of $X$ supported on vertices $v$ with $\rho(v) < n$. The \emph{descending link} $\lk^{\downarrow} v$ of a vertex $v$ is the full subcomplex of $\lk v$ of those vertices $w$ with $\rho(w) \le \rho(v)$ and the \emph{descending star} $\st^{\downarrow}$ is defined analogously. That $\rho$ is a Morse function implies that that the inequality $\rho(w) \le \rho(v)$ is strict for the descending link and for the descending star is not strict only when $w = v$. In particular, the descending star is the cone over the descending link.

The goal of combinatorial Morse theory is to compare the connectivity properties of sublevel sets to each other and to those of $X$. The tool to do so is a basic lemma, called the Morse Lemma:

\begin{lemma}
\label{lem:morse}
Let $\rho$ be a Morse function on $X$. Let $m \le n \le \infty$ and assume that for every vertex $v$ with $m \le \rho(v) < n$ the descending link of $v$ is $(d-1)$-connected. Then the pair $(X_{\rho <n}, X_{\rho < m})$ is $d$-connected, that is, $\pi_k(X_{\rho < m} \to X_{\rho < n})$ is an isomorphism for $k < d$ and an epimorphism for $k = d$.
\end{lemma}

\begin{proof}
The basic observations are that
\[
X_{\rho < m+1} = X_{\rho < m} \cup \bigcup_{\rho(v) = m} \st^{\downarrow} v\text{,}
\]
that $\st^\downarrow v \cap \st^\downarrow v' \subseteq X_{\rho< m}$ for $\rho(v) = m = \rho(v')$, and that $\st^\downarrow v \cap X_{\rho < m} = \lk^\downarrow v$. As a consequence (using compactness of spheres) it suffices to study the extension $Y \defeq X_{\rho < m} \cup_{\lk^\downarrow v} \st^\downarrow v$ for an individual vertex $v$ with $\rho(v) = m$.

In this situation $\pi_k(Y, X_{\rho<m}) \cong \pi_k(\st^\downarrow v, \lk^\downarrow v)$ for $k \le d$. This can be seen by separately looking at $\pi_1$ and $H_*$ (where excision holds) and applying Hurrewicz's theorem \cite[Theorem~4.37]{hatcher02}. The statement now follows from the long exact homotopy/homomlogy sequence for the pair $(\st^\downarrow v, \lk^\downarrow v)$.
\end{proof}


\subsection{Finiteness properties of fundamental groups of Ore categories}
\label{sec:proof_scheme}

We take up the construction from Section~\ref{sec:space_ore_garside}. So $\cat$ is again a right-Ore category, $\gars$ is a left- or right-Garside family closed under factors, and $* \in \Ob(\cat)$ is a base object. More than requiring strong Noetherianity, we now need a height function $\rho \colon \Ob(\cat) \to \N$.

We use these data and assumptions to provide a criterion to prove finiteness properties for the fundamental group.

We need to introduce one further space construction. It is another variant of the nerve construction. For $x \in \Ob(\cat)$ let $E(x)$ be the set of equivalence classes in $a \in \eltry(-,x) \smallsetminus \eltry^\times(x,x)$ modulo the equivalence relation that $\bar{a} = \bar{a\smash'}$ if there exists a $g \in \cat^\times$ with $ga = a'$. We define a relation $\le$ on $E(x)$ by declaring $\bar{a} \le \bar{b}$ if there is an $f \in \cat$ with $fa = b$. Note that if $g$ and $f$ as above exist, they lie in $\eltry$ so the description can be formulated purely in terms of $\eltry$. As in Lemma~\ref{lem:poset_contractible} one sees that $\le$ is a partial order on $E(x)$, however it is usually not contractible.

\begin{theorem}
Let $\cat$ be a right-Ore category and let $* \in \Ob(\cat)$. Let $\gars$ be a locally finite left- or right-Garside family that is closed under factors. Let $\rho \colon \Ob(\cat) \to \N$ be a height function such that $\{x \in \Ob(\cat) \mid \rho(x) \le n\}$ is finite for every $n \in \N$.
\label{thm:generic_proof}
Assume
\begin{enumerate}[align=left,leftmargin=*,widest=(\textsc{stab})]
\item[\mlabel{stab}{\textsc{stab}}] $\cat^\times(x,x)$ is of type $F_n$ for all $x$,
\item[\mlabel{lk}{\textsc{lk}}] $\abs{E(x)}$ is $(n-1)$-connected for all $x$ with $\rho(x)$ beyond a fixed bound.
\end{enumerate}
(If $\rho$ is unbounded on the component of $*$ then it suffices if \eqref{stab} holds for every $x$ with $\rho(x)$ beyond a fixed bound.)

Then $\pi_1(\cat,*)$ is of type $F_n$.
\end{theorem}

\begin{remark}
Recall that $\cat$ can be replaced by the component of $*$ in $\cat$ so all assumptions need to be made only for that component.
 \end{remark}
 
\begin{proof}
We take $X$ to be the complex constructed in Section~\ref{sec:finiteness_properties}. Assume first that \eqref{stab} holds for all $x \in \Ob(\cat)$.

For a vertex $\bar{a} \in X$ with $a \in \Ore(\cat)(*,x)$ we define $\rho(\bar{a}) = \rho(x)$. This is a $\pi_1(\cat,*)$-invariant Morse function which we think of as height. For $n \in \N$ we consider the subcomplex $X_{\rho < n}$ supported on vertices of height $< n$.

We want to see that every $X_{\rho < n}$ is $\pi_1(\cat,*)$-cocompact. To do so we note that $\pi_1(\cat,*)$ acts transitively on vertices $\bar{a}$ with $a \in \Ore(\cat)(*,x)$: indeed, if $\bar{b}$ is another such then $ba^{-1} \in \pi_1(\cat,*)$ takes $\bar{a}$ to $\bar{b}$. It follows from the assumption on $\rho$ that there are only finitely many vertices $\bar{a}$ with $\rho(\bar{a}) < n$ up to the $\pi_1(\cat,*)$-action. Cocompactness now follows from Observation~\ref{obs:loc_fin}.

Stabilizers are of type $F_n$ by Lemma~\ref{lem:stabilizers} because finiteness properties are inherited by finite-index subgroups.

Let $N$ be large enough so that all the $x \in \Ob(\cat)$ for which the nerve of $\abs{E(x)}$ is not $(n-1)$-connected have $\rho(x) < N$. We have just seen that $\pi_1(\cat,*)$ acts on $X_{\rho < N}$ cocompactly with stabilizers of type $F_n$, so once we show that $X_{\rho < N}$ is $(n-1)$-connected, we are done by Theorem~\ref{thm:browns_criterion}. We want to apply the Morse Lemma (Lemma~\ref{lem:morse}) so let us look at the descending link of a vertex $\bar{b}$ of $X$, where $b \in \cat(*,x)$. The vertices in the descending link are the $\bar{a}$ that are comparable with $\bar{b}$ and have $\rho(\bar{a}) < \rho(\bar{b})$. The condition on the height shows that $a$ cannot be a right-multiple of $b$ but has to be a left-factor. Thus $a^{-1}b \in \eltry(-,x)$ and the descending link of $\bar{b}$ is the realization of $\{\bar{a} \mid a \prec b\}$. We see that the map $\eltry(-,x) \smallsetminus \eltry(x,x) \to \{\bar{a} \mid a \prec b\}$ that takes $f$ to $\overline{a f^{-1}}$ is an order-reversing surjection. The definition of $E(x)$ is made so that the induced map $E(x) \to \{\bar{a} \mid a \prec b\}$ is well-defined and a order-reversing bijection. Since $\abs{E(x)}$ is $(n-1)$-connected by assumption, this completes the proof in the case that \eqref{stab} holds for all $x$.

If \eqref{stab} only holds for $x$ with $\rho(x) \ge M$ let $*'$ be in the component of $*$ satisfying $\rho(*') > M$. Since $\cat$ is Ore, one sees that
\[
\pi_1(\cat,*) = \pi_1(\cat,*') = \pi_1(\cat_{\rho \ge M}, x_0)
\]
where $\cat_{\rho \ge M}$ is obtained from $\cat$ by removing objects $y$ with $\rho(y) < M$. Moreover local finiteness of $\gars$ implies that the complexes $E(y)$ for $\cat$ and for $\cat_{\rho \ge r}$ are the same for $y$ in the component of $*'$ once $\rho(y)$ is large enough. One can therefore consider $\cat_{\rho \ge M}$ instead of $\cat$ with the effect that the groups $\cat^\times(x,x)$ only need to be of type $F_n$ when $\rho(x) \ge M$.
\end{proof}

\begin{corollary}
\label{cor:generic_proof}
Let $\cat$, $\gars$, $\rho$, $*$ be as in the theorem.
If $\cat^\times(x,x)$ is of type $F_\infty$ for every $x$ and the connectivity of $\abs{E(x)}$ tends to infinity for $\rho(x) \to \infty$ then $\pi_1(\cat,*)$ is of type $F_\infty$.
\end{corollary}

The construction of $X$ uses two important ideas. One is the passage from $\abs{P}$ to $X$ which is due to Stein, see \cite[Theorem~1.5]{stein92}. The other is to take $P$ to consist of $\cat^\times$-equivalence classes and goes back to \cite{bux16}. Apart from these ideas the main difficulty in proving that $\pi_1(\cat,*)$ is of type $F_n$ lies in establishing the connectivity properties of the complexes $\abs{E(x)}$. This problem depends individually on the concrete setup and we will see various examples later.


\subsection{Example: $F$ is of type $F_\infty$}

As a first illustration of the results in this section we reprove a result due to Brown and Geoghegan~\cite{brown84}:

\begin{proposition}
\label{prop:f_finfty}
Thompson's group $F$ is of type $F_\infty$.
\end{proposition}

We have seen in Proposition~\ref{prop:thomcat_ore} that $\thomcat$ is right-Ore and admits a height function and by Corollary~\ref{cor:f_garside_family} it has a locally finite left-Garside family that is closed under factors. Moreover, $\thomcat^\times(x,x) = \{1_x\}$ for every $x$, so \eqref{stab} is satisfied as well. It only remains to verify \eqref{lk}. Although things are not always as easy, we remark that this is the typical situation: property \eqref{lk} is where one actually needs to show something.

To understand the complexes $\abs{E(n)}$ we first need to unravel the definition. Recall that a \emph{matching} of a graph $\Gamma$ is a set of edges $M \subseteq E(\Gamma)$ that are pairwise disjoint. Matchings are ordered by containment and we denote the poset of matchings by $\match(\Gamma)$. In fact, since every subset of a matching is again a matching, $\match(\Gamma)$ is (the face poset of) a simplicial complex, the \emph{matching complex}. We denote by $L_n$ the \emph{linear graph} on $n$ vertices $\{1,\ldots,n\}$ so its edges are $\{i,i+1\}$ for $1 \le i <n$.

\begin{lemma}
\label{lem:link_match_linear}
The poset $E_\thomcat(n)$ is isomorphic to $\match(L_n)$.
\end{lemma}

\begin{proof}
Let $f \in \eltry_\thomcat(-,n)$, so $f$ is an element of $E_\thomcat(n)$. We identify the roots of $f$ with the vertices of the linear graph $L_n$ on the vertices $\{1,\ldots,n\}$. Every caret of $f$ connects two of these roots and thus corresponds to an edge of $L_n$. All these edges are disjoint so the resulting subgraph $M_f$ of $L_n$ is a matching. It is clear that conversely every matching of $L_n$ arises in a unique way from an elementary forest.

If $h \le f$ then $h$ is a left-multiple of $f$, that is, $f$ can be obtained from $h$ by adding carets to some roots of $h$ that do not have carets yet. On the level of graphs this means that $M_f$ is obtained from $M_h$ by adding edges so that $M_h \le M_f$ in the poset of matchings.
\end{proof}

\begin{remark}
In particular, $E_\thomcat(n)$ is (the face poset of) a simplicial complex. Note that the realization as a poset is the barycentric subdivision of the realization as a simplicial complex, and in particular both are homeomorphic. So there is no harm in working with the coarser cell structure where elements of $E_\thomcat(n)$ are simplices rather than vertices. This fact applies in most of our cases.
\end{remark}

Matching complexes of various graphs have been studied intensely and their connectivity properties can be verified in various ways \cite{bjoerner94}. In fact, for linear and cyclic graphs the precise homotopy type is known \cite[Proposition~11.16]{kozlov08}.

Rather than using the known optimal connectivity bounds we use the opportunity to introduce a criterion due to Belk and Forrest \cite[Theorem~4.9]{belk} that is particularly well suited to verifying that the connectivity of the spaces $E(x)$ tends to infinity in easier cases. We need to introduce some notation.

An abstract simplicial complex $X$ is \emph{flag} if every set of pairwise adjacent vertices forms a simplex. A simplex $\sigma$ in a simplicial flag complex is called a \emph{$k$-ground} for $k \in \N$ if every vertex of $X$ is connected to all but at most $k$ vertices of $\sigma$. The complex is said to be \emph{$(n,k)$-grounded} if there is an $n$-simplex that is a $k$-ground.

\begin{theorem}
\label{thm:belk_forrest}
For $m,k \in \N$ every $(mk,k)$-grounded flag complex is $(m-1)$-connected.
\end{theorem}

The reference requires $m,k \ge 1$ but it is clear that every $(0,k)$-grounded flag complex is non-empty, and every $(0,0)$-grounded flag complex is a cone an therefore contractible.

Using Theorem~\ref{thm:belk_forrest} we verify:

\begin{lemma}
\label{lem:match_connectivity}
For every $n \in \N$ let $\Gamma_n$ be a subgraph of $K_n$ containing $L_n$.
The connectivity of $\match(\Gamma_n)$ goes to infinity as $n$ goes to infinity.
\end{lemma}

\begin{proof}
Consider the matchings of $L_n$ that use only the edges $\{2i-1, 2i\}$, $1 \le i \le \floor{n/2}$. They form an $(\floor{n/2}-1)$-simplex $\sigma$ in $\match(\Gamma_n)$. If $v = \{j,k\}$ is any edge of $\Gamma_n$, so a vertex of $\match(\Gamma_n)$ then there are at most $2$ vertices of $\sigma$ that $v$ is not connected to: one is $\{j-1,j\}$ or $\{j,j+1\}$, the other is $\{k-1,k,\}$ or $\{k,k+1\}$. This shows that $\match(\Gamma_n)$ is $(\floor{n/2}-1,2)$-grounded, so by Theorem~\ref{thm:belk_forrest} it is $(\floor{n/4}-1)$-connected.
\end{proof}

\begin{proof}[Proof of Proposition~\ref{prop:f_finfty}]
We want to apply Corollary~\ref{cor:generic_proof}. The only thing left to check is condition \eqref{lk}. This follows from Lemmas~\ref{lem:link_match_linear} and~\ref{lem:match_connectivity}.
\end{proof}


\section{The indirect product of two categories}
\label{sec:indirect_product}

The construction introduced in this section will help us to produce more interesting examples. It is usually called the Zappa--Szép product in the literature of groups and monoids, cf.~\cite{brin05}. The Zappa--Szép product naturally generalizes the semidirect product in the same way as the semidirect product generalizes the direct product. We think that such a basic construction should have a simpler name and therefore call it the \emph{indirect product}.

For motivation, let $M$ be a monoid (or group) whose multiplication we denote by $\circ$ and suppose that $M$ decomposes uniquely as $M = A \circ B$. By this we mean that $A$ and $B$ are submonoids of $M$ such that every element $m \in M$ can be written in a unique way as $m = a \circ b$ with $a \in A$ and $b \in B$. In particular, if $b' \in B$ and $a' \in A$, the product $m = b' \circ a'$ can be rewritten as $b'\circ a'= a\circ b$. This allows to formally define maps $A \times B \to B, (a',b') \mapsto a' \cdot b' \defeq a$ and $B \times A \to A, (b',a') \mapsto {b'}^{a'} \defeq b$ so that
\[
b \circ a = (b \cdot a) \circ b^a\text{.}
\]
These maps turn out to be actions of monoids on sets. If both actions are trivial then $M$ is a direct product, if one of the actions is trivial then $M$ is a semidirect product, and in general it is an indirect product.

We therefore start by introducing the appropriate notion of actions of categories.

\subsection{Actions}

Let $\cat$ be a category and let $(X_m)_{m \in \Ob(\cat)}$ be a family of sets, one for each object of $\cat$. We say that a left action of $\cat$ on $(X_m)_m$ is a family of maps
\begin{align*}
\cat(n,m) \times X_m &\to X_n\\
(f,s) &\mapsto f \cdot s
\end{align*}
satisfying $1_m \cdot s = s$ for all $m \in \Ob(\cat)$ and $s \in X_m$ and $fg \cdot s = f \cdot (g \cdot s)$ whenever $fg$ is defined. A right action is defined analogously. An action is said to be \emph{injective} if $f \cdot x = f \cdot y$ implies $x = y$. Note that actions of groupoids are always injective.

In our examples the family $(X_m)_m$ itself will consist of morphisms of a category with the same objects as $\cat$. We have to bear in mind, however, that the action is on these as sets and does not preserve products.

\subsection{The indirect product}

Let $\cat$ be a category and let $\thomcat$ and $\grpd$ be subcategories. We say that $\cat$ is an \emph{internal indirect product} $\thomcat \bowtie \grpd$ if every $h \in \cat$ can be written in a unique way as $h = fg$ with $f \in \thomcat$ and $g \in \grpd$. Note that this means in particular that $\Ob(\cat) = \Ob(\thomcat) = \Ob(\grpd)$. Given elements $f \in \thomcat(x,-)$ and $g \in \grpd(-,x)$ there exist then unique elements $f' \in \thomcat$ and $g' \in \grpd$ such that $gf = f'g'$, see Figure~\ref{fig:zappa-szep_gen}. In this situation we define $g \cdot f$ to be $f'$ and $g^f$ to be $g'$.

The following properties are readily verified, see Figure~\ref{fig:zappa-szep_ones} and~\ref{fig:zappa-szep_prod}, the last four hold whenever one of the sides is defined:
\begin{enumerate}[label=(IP\arabic{*}), ref=IP\arabic{*},leftmargin=*]
\item $1_x \cdot f = f$ for $f \in \thomcat(x,-)$,\label{item:zappa-szep_1_acts_left}
\item $g^{1_y} = g$ for $g \in \grpd(-,y)$,\label{item:zappa-szep_1_acts_right}
\item $(g_1g_2) \cdot f = g_1 \cdot (g_2 \cdot f)$,\label{item:zappa-szep_product_acts_left}
\item $g^{f_1f_2} = (g^{f_1})^{f_2}$,\label{item:zappa-szep_product_acts_right}
\item $1_x^f = 1_y$ for $f \in \thomcat(x,y)$,\label{item:zappa-szep_1_is_acted_left}
\item $g \cdot 1_y = 1_z$ for $g \in \grpd(z,x)$,\label{item:zappa-szep_1_is_acted_right}
\item $(g_1g_2)^f = g_1^{(g_2 \cdot f)}g_2^f$,\label{item:zappa-szep_product_is_acted_left}
\item $g \cdot (f_1f_2) = (g \cdot f_1)(g_2^{f_1} \cdot f_2)$.\label{item:zappa-szep_product_is_acted_right}
\end{enumerate}

\begin{figure}

\subfloat[]{
\begin{tikzpicture}[yscale=-1,scale=1.5]
\node (z) at (0,0) {$\bullet$};
\node (x) at (1,0) {$\bullet$};
\node (y) at (1,1) {$\bullet$};
\node (w) at (0,1) {$\bullet$};
\path[<-]
(z) edge node[above]{$g$} (x)
(x) edge node[right]{$f$} (y);
\path[dashed,<-]
(w) edge node[below]{$g^f$} (y)
(z) edge node[left]{$g \cdot f$} (w);
\end{tikzpicture}
\label{fig:zappa-szep_gen}
}
\hfill
\subfloat[]{
\begin{tikzpicture}[yscale=-1,scale=1.5]
\node (z) at (0,0) {$x$};
\node (x) at (1,0) {$x$};
\node (y) at (1,1) {$y$};
\node (w) at (0,1) {$y$};
\path[<-]
(z) edge node[above]{$1_x$} (x)
(x) edge node[right]{$f$} (y)
(w) edge node[below]{$1_y$} (y)
(z) edge node[left]{$f$} (w);
\end{tikzpicture}
\label{fig:zappa-szep_ones}
}
\hfill
\subfloat[]{
\begin{tikzpicture}[yscale=-1,scale=1.5]
\node (z) at (0,0) {$\bullet$};
\node (x) at (1,0) {$\bullet$};
\node (y) at (1,1) {$\bullet$};
\node (w) at (0,1) {$\bullet$};
\node (u) at (-1,0) {$\bullet$};
\node (v) at (-1,1) {$\bullet$};
\path[<-]
(u) edge node[above]{$g_1$} (z)
(v) edge node[below]{$g_1^{g_2 \cdot f}$} (w)
(u) edge node[left]{$g_1 \cdot (g_2 \cdot f)$} (v)
(z) edge node[above]{$g_2$} (x)
(x) edge node[right]{$f$} (y)
(w) edge node[below]{$g_2^f$} (y)
(z) edge node[right]{$g_2 \cdot f$} (w);
\end{tikzpicture}
\label{fig:zappa-szep_prod}
}
\caption{The indirect product.}
\label{fig:zappa-szep}
\end{figure}
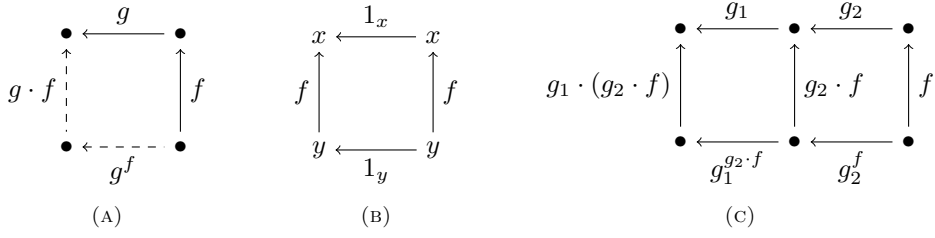

The first four relations say that the map $(g,f) \mapsto g \cdot f$ is an left action of $\grpd$ on the sets $(\thomcat(x,-))_x$ and that $(g,f) \mapsto g^f$ is a right action of $\thomcat$ on the sets $(\grpd(-,y))_y$. The next two relations say that identity elements are taken to identity elements, while the last two are cocycle conditions. We call actions satisfying \eqref{item:zappa-szep_1_acts_left} to \eqref{item:zappa-szep_product_is_acted_right} \emph{indirect product actions}.

Now assume that conversely categories $\thomcat$ and $\grpd$ with $\Ob(\thomcat) = \Ob(\grpd)$ are given together with indirect product actions of $\thomcat$ and $\grpd$ on each other. Then the \emph{external indirect product} $\cat = \thomcat \bowtie \grpd$ is defined to have objects $\Ob(\cat) = \Ob(\thomcat) = \Ob(\grpd)$ and morphisms
\[
\cat = \bigcup_{x \in \Ob(\cat)} \{ (f,g) \mid f \in \thomcat(-,x), g \in \grpd(x,-)\}\text{.}
\]
Composition is defined by
\begin{equation}
(f_1,g_1)(f_2,g_2) = (f_1(g_1 \cdot f_2),g_1^{f_2}g_2)\text{.}\label{eq:zappa-szep_definition}
\end{equation}

\begin{figure}
\centering
\scriptsize
\begin{tikzpicture}[yscale=-1,scale=2]
\colorlet{darkgreen}{green!65!black}
\node (w) at (0,0) {$\bullet$};
\node (x) at (1,1) {$\bullet$};
\node (y) at (2,2) {$\bullet$};
\node (z) at (3,3) {$\bullet$};

\node (r) at (0,1) {$\bullet$};
\node (s) at (1,2) {$\bullet$};
\node (t) at (2,3) {$\bullet$};

\node (o) at (0,2) {$\bullet$};
\node (p) at (1,3) {$\bullet$};

\node (m) at (0,3) {$\bullet$};
\path[<-]
(r) edge node[above]{$g_1$} (x)
(x) edge node[right]{$f_2$} (s)
(s) edge node[above]{$g_2$} (y)
(y) edge node[right]{$f_3$} (t);
\path[<-]
(w) edge[ultra thick,darkgreen] (r)
(t) edge[ultra thick,darkgreen] (z);
\path[<-]
(w) edge[ultra thick,dashed] node[right]{$f_1$} (r)
(t) edge[ultra thick,dashed] node[above]{$g_3$} (z);
\path[<-]
(r) edge[ultra thick,darkgreen] node[fill=white]{$g_1 \cdot f_2$} (o)
(o) edge node[above]{$g_1^{f_2}$} (s)
(s) edge node[fill=white]{$g_2 \cdot f_3$} (p)
(p) edge[ultra thick] node[above]{$g_2^{f_3}$} (t);
\path[<-]
(o) edge node[fill=white,shift={(.5,.2)}]{$g_1^{f_2} \cdot (g_2 \cdot f_3)$} (m)
(m) edge node[below]{$(g_1^{f_2})^{g_2 \cdot f_3}$} (p);
\path[<-]
(r) edge[bend left=90,ultra thick,dashed] node[left]{$g_1 \cdot (f_2(g_2\cdot f_3))$} (m);
\path[<-]
(m) edge[bend left=90,ultra thick,darkgreen] node[below]{$(g_1^{f_2}g_2)^{f_3}$} (t);
\path[<-]
(o) edge[bend left=90,ultra thick,darkgreen] node[fill=white,shift={(-.1,-.3)}]{$(g_1^{f_2}g_2) \cdot f_3$} (m);
\path[<-]
(m) edge[bend left=90,ultra thick,dashed] node[fill=white,shift={(.5,-.05)}]{$g_1^{f_2(g_2 \cdot f_3)}$} (p);
\end{tikzpicture}
\caption{Associativity in $\thomcat \bowtie \grpd$. The dashed thick black and the thick green path are the components of $(f_1,g_1)((f_2,g_2)(f_3,g_3))$ and $((f_1,g_1)(f_2,g_2))(f_3,g_3)$ respectively.}
\label{fig:zappa-szep_associative}
\end{figure}

\begin{lemma}
The external indirect product $\thomcat \bowtie \grpd$ is well-defined. It is naturally isomorphic to the internal indirect product of the copies of $\thomcat$ and $\grpd$ inside $\thomcat \bowtie \grpd$.
\end{lemma}

\begin{proof}
That the identity morphisms $(1_x,1_x)$ behave as they should is easily seen using relations \eqref{item:zappa-szep_1_acts_left}, \eqref{item:zappa-szep_1_acts_right}, \eqref{item:zappa-szep_1_is_acted_left}, and \eqref{item:zappa-szep_1_is_acted_right}. To check associativity we verify the four equations
\begin{gather}
g_1^{f_2(g_2 \cdot f_3)} \stackrel{\eqref{item:zappa-szep_product_acts_right}}{=} (g_1^{f_2})^{g_1 \cdot f_3} \label{eq:zappa-szep_associative_1}\mathrlap{,}\\
(g_1^{f_2}g_2) \cdot f_3 \stackrel{\eqref{item:zappa-szep_product_acts_left}}{=} g_1^{f_2} \cdot (g_2 \cdot f_3) \label{eq:zappa-szep_associative_2}\mathrlap{,}\\
g_1^{f_2(g_2 \cdot f_3)}g_2^{f_3} \stackrel{\eqref{eq:zappa-szep_associative_1}}{=} (g_1^{f_2})^{g_1 \cdot f_3} g_2^{f_3} \stackrel{\eqref{item:zappa-szep_product_is_acted_left}}{=}(g_1^{f_2}g_2)^{f_3}\label{eq:zappa-szep_associative_3}\mathrlap{,\quad\text{and}}\\
(g_1 \cdot f_2)((g_1^{f_2}g_2) \cdot f_3) \stackrel{\eqref{eq:zappa-szep_associative_2}}{=} (g_1 \cdot f_2)(g_1^{f_2} \cdot (g_2 \cdot f_3)) \stackrel{\eqref{item:zappa-szep_product_is_acted_right}}{=} g_1 \cdot (f_2(g_2 \cdot f_3))\label{eq:zappa-szep_associative_4}
\end{gather}
see Figure~\ref{fig:zappa-szep_associative}.

The categories $\thomcat$ and $\grpd$ naturally embed into the external indirect product $\thomcat \bowtie \grpd$ as $f \mapsto (f, 1_y)$ for $f \in \thomcat(-,y)$ and $g \mapsto (1_x,g)$ for $g \in \grpd(x,-)$. Any morphism of $\thomcat \bowtie \grpd$ decomposes as $(f,g) = (f,1_y)(1_y,g)$ and it is clear from \eqref{eq:zappa-szep_definition} that respective actions on each other are the ones used to define $\thomcat \bowtie \grpd$.
\end{proof}

If the action of $\grpd$ on $\thomcat$ is trivial then the indirect product is a \emph{semidirect product} $\thomcat \ltimes \grpd$. Similarly, if the action of $\thomcat$ on $\grpd$ is trivial then it is a semidirect product $\thomcat \rtimes \grpd$. Finally, if both actions are trivial then the indirect product is in fact a \emph{direct product} $\thomcat \times \grpd$.

We close the section by collecting facts that ensure that an indirect product is Ore.

\begin{lemma}
\label{lem:zappa-szep_cancellative}
If $\thomcat$ and $\grpd$ are right-cancellative and the action of $\thomcat$ on $\grpd$ is injective then $\thomcat \bowtie \grpd$ is right-cancellative.
Symmetrically, if $\thomcat$ and $\grpd$ are left-cancellative and the action of $\grpd$ on $\thomcat$ is injective then $\thomcat \bowtie \grpd$ is left-cancellative.
\end{lemma}

\begin{proof}
If $f_1g_1fg = f_2g_2fg$ then $f_1(g_1 \cdot f) = f_2(g_2 \cdot f)$ and $g_1^fg = g_2^fg$. Since $\grpd$ is right-cancellative the latter equation shows that $g_1^f = g_2^f$ and injectivity of the action then implies $g_1 = g_2$. Putting this in the former equation and using right-cancellativity of $\thomcat$ gives $f_1 = f_2$.
\end{proof}

\begin{observation}
\label{lem:zappa-szep_common_multiples}
Let $\thomcat$ have common right-multiples and let $\grpd$ be a groupoid. Then $\thomcat \bowtie \grpd$ has common right-multiples.
\end{observation}

\begin{proof}
Let $fg \in \thomcat \bowtie \grpd$ with $f \in \thomcat$ and $g \in \grpd$. Since $\grpd$ is a groupoid, $f$ is both a left-factor and a right-multiple of $fg$. It follows that common right-multiples exist in $\thomcat \bowtie \grpd$ because they exist in $\thomcat$.
\end{proof}

\begin{observation}
\label{lem:zappa-szep_invertibles}
Let $\thomcat$ have no non-trivial invertible morphisms and let $\grpd$ be a groupoid. Then $(\thomcat \bowtie \grpd)^\times = \grpd$.
\end{observation}

\begin{proposition}
\label{prop:zappa-szep_conditions}
Let $\cat = \thomcat \bowtie \grpd$ where $\thomcat$ has no non-trivial invertibles and $\grpd$ is a discrete groupoid.
\begin{enumerate}
\item If $\thomcat$ is right-Ore and the action of $\thomcat$ on $\grpd$ is injective then $\cat$ is right-Ore.\label{item:product_ore}
\item If $\thomcat$ is strongly Noetherian then so is $\cat$.\label{item:product_noeth}
\item If $\rho$ is a height function on $\thomcat$ then it is a height function on $\cat$.\label{item:product_height}
\item If $\gars$ is a left-Garside family in $\thomcat$ then it is a left-Garside family in $\cat$.\label{item:product_left_gars}
\item If $\gars$ is a right-Garside family in $\thomcat$ then $\gars\grpd$ is a right-Garside family in $\cat$.\label{item:product_right_gars}
\end{enumerate}
\end{proposition}

\begin{proof}
Property~\eqref{item:product_ore} follows from Lemma~\ref{lem:zappa-szep_cancellative} and Observation~\ref{lem:zappa-szep_common_multiples}.
Properties~\eqref{item:product_noeth} and~\eqref{item:product_left_gars} follow from the fact that for $f \in \thomcat$ and $g \in \grpd$ the morphisms $f$ and $fg$ are right-multiples by invertibles of each other. Property~\ref{item:product_height} follows from $\grpd$ being discrete (i.e.\ every morphism being an endomorphism). Toward \eqref{item:product_right_gars} it is clear that every right-factor of $\gars\grpd$ is contained in $\gars\grpd$. Moreover if $t$ is an $\gars$-tail for $f$ then $tg$ is a $\gars$-tail for $fg$.
\end{proof}


\section{Examples: categories constructed by indirect products}
\label{sec:indirect_product_examples}

In this section we show how the indirect product can be used to construct new groups. The basic examples are Thompson's groups $T$ and $V$ as well as the braided Thompson groups, which all arise as fundamental groups of categories of the form $\thomcat \bowtie \grpd$ where $\grpd$ is an appropriate groupoid. More generally, the groups studied in joint work with Zaremsky \cite{witzel16a} are essentially by definition groups that can be obtained in this form. Later we also describe other groups obtained via indirect products.

We will sometimes draw pictures to motivate our definition. In these pictures the up-direction always corresponds to left in our notation and down corresponds to right. This is especially relevant for group elements. For example, a permutation $X \from X, g(x)$ \reflectbox{$\mapsto$} $x$ will be depicted by connecting the point $x$ at the bottom to the point $g(x)$ at the top.

\subsection{Thompson's groups $\mathbf{T}$ and $\mathbf{V}$}
\label{sec:tcat_vcat}

In this section we introduce Thompson's groups $T$ and $V$ as fundamental groups of categories $\Tcat$ and $\Vcat$. The categories will be obtained from $\thomcat$ as indirect products with groupoids and we start by introducing these.

We define $\grpd_T$ and $\grpd_V$ to be groupoids whose objects are positive natural numbers with $\grpd_T(m,n) = \emptyset$ for $m \ne n$. We put $\grpd_T(n,n) = \Z/n\Z$ and $\grpd_V(n,n) = \Sym_n$. We want to define $\Tcat = \thomcat \bowtie \grpd_T$ and $\Vcat = \thomcat \bowtie \grpd_V$ and have to specify the actions that define these indirect products. That is, given a forest $f \in \thomcat(m,n)$ and a group element $g \in \grpd(m,m)$ we need to specify how the product $gf$ should be written as $(g \cdot f)g^f$ with $g \cdot f \in \thomcat(m,n)$ and $g^f \in \grpd(n,n)$ (for $\grpd$ one of $\grpd_T$ and $\grpd_V$).

\begin{figure}[htb]
\centering
\newcommand{\treea}[1]{
\begin{scope}[#1]
\node[end] (root) at (0,0) {};
\node[end] (left) at (-.5,.5) {};
\node[end] (right) at (.5,.5) {};
\draw (left) -- (root) -- (right);
\end{scope}
}
\begin{tikzpicture}[scale=1,yscale=-1]
\colorlet{darkgreen}{green!65!black}
\foreach \x in {0,1}
\draw (\x,1) -- (\x+1,0);
\draw[thick, darkgreen] (2,1) -- (2.5,.5)
	(-.5,.5) -- (0,0);
\treea{shift={(2,1)}}
\foreach \x in {0,1,2}
\node[end] at (\x,0) {};
\foreach \x in {0,1,2}
\node[end] at (\x,1) {};
\foreach \x in {0,1}
\draw (\x,1) -- (\x,1.5);
\foreach \x in {0,1}
\node[end] at (\x,1.5) {};
\node at (3,1) {=};
\begin{scope}[xshift=4cm]
\begin{scope}[xshift=0cm]
\node[end] (root) at (.5,.5) {};
\node[end] (left) at (0,1) {};
\node[end] (right) at (1,1) {};
\draw (left) -- (root) -- (right);
\end{scope}
\foreach \x in {2,3}
\draw (\x,.5) -- (\x,1);
\foreach \x in {2,3}
\node[end] at (\x,.5) {};
\begin{scope}[yshift=1cm]
\foreach \x in {0,1}
\draw (\x,1) -- (\x+2,0);
\foreach \x in {2,3}
\draw[thick, darkgreen] (\x,1) -- (3.5,1-1/2*3.5+1/2*\x)
	(-.5,1-1/2*3.5+1/2*\x) -- (\x-2,0);
\foreach \x in {0,1,2,3}
\node[end] at (\x,0) {};
\foreach \x in {0,1,2,3}
\node[end] at (\x,1) {};
\end{scope}
\end{scope}
\end{tikzpicture}
\caption{Defining $\thomcat \bowtie \grpd_T$. The picture shows how to write $gf$ as $(g \cdot f)g^f$ in the case where $f$ is the caret $\lambda^3_3 \in \thomcat(3,4)$ and $g = 1 + \Z/3\Z \in \grpd_T(3,3)$. The colored strand gets doubled under the action of $f$. As a result $g \cdot f = \lambda^3_1 \in \thomcat(3,4)$ and $g^f = 2 + \Z/4\Z \in \grpd_T(4,4)$.}
\label{fig:t_zappa_szep}
\end{figure}
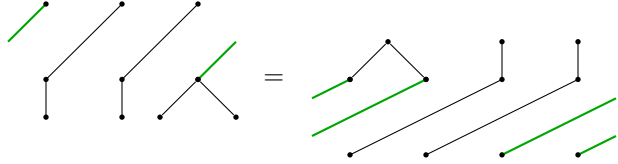

Since $\grpd_T$ is contained in $\grpd_V$, it would suffice to only define the actions for $\grpd_V$, but we look at the simpler case of $\grpd_T$ first.

We need to rewrite a cyclic permutation followed by a tree as a tree followed by a cyclic permutation. This is illustrated in Figure~\ref{fig:t_zappa_szep}. For $f \in \thomcat(m,n)$ and $g = \ell + \Z/m\Z \in \grpd_T(m,m)$ the forest $g \cdot f$ is just $f$ with the trees rotated by $\ell$ to the right. The definition of $g^f$ is more subtle: looking at the figure we see that we have to define it to be $k + \Z/n\Z$ where $k$ is the number of leaves of the first $\ell$ trees of $g \cdot f$, or equivalently, to be the number leaves of the last $\ell$ trees of $f$. Note that this number does not depend on the chosen representative $\ell$: if we replace $\ell$ by $\ell + m$, instead of $k$ we get $k+n$, because we counted every leaf once more. If $k_\ell$ denotes the number of leaves of the last $\ell$ trees of $f$, the sequence $(k_\ell)_{0 \le \ell < m}$ is strictly increasing. This shows:

\begin{observation}
\label{obs:tcat_injective}
The action of $\thomcat$ on $\grpd_T$ is injective.
\end{observation}

\begin{lemma}
The actions of $\thomcat$ and $\grpd_T$ on each other are indirect product actions.
\end{lemma}

\begin{proof}
Conditions \eqref{item:zappa-szep_1_acts_left}, \eqref{item:zappa-szep_1_acts_right}, \eqref{item:zappa-szep_product_acts_left}, \eqref{item:zappa-szep_product_acts_right}, \eqref{item:zappa-szep_1_is_acted_left}, \eqref{item:zappa-szep_1_is_acted_right} are clear.

The condition \eqref{item:zappa-szep_product_is_acted_left} in our setting follows from the fact that the last $k+\ell$ trees of $f$ are the last $\ell$ trees of $f$ plus the last $k$ trees of $(\ell + m\Z) \cdot f$. Condition \eqref{item:zappa-szep_product_is_acted_right} can be verified drawing a picture.
\end{proof}

The lemma allows us to define $\Tcat = \thomcat \bowtie \grpd_T$. Combining Observation~\ref{obs:tcat_injective} with Propositions~\ref{prop:thomcat_ore} and~Corollary~\ref{cor:f_garside_family}  and applying Proposition~\ref{prop:zappa-szep_conditions} we find:

\begin{corollary}
\label{cor:tcat_righ_ore}
The category $\Tcat$ is right-Ore and admits a height function and a left-Garside family $\gars$ that is closed under factors such that $\gars(x,-)/\gars^\times$ is finite for every $x$.
\end{corollary}

The fundamental group $\pi_1(\Tcat,1)$ is \emph{Thompson's group $T$}.

Now we want to define the actions of $\thomcat$ and $\grpd_V$ on each other. So let $f \in \thomcat(m,n)$ and let $g \in \grpd_V(m,m)$. The action of $\grpd_V$ on $\thomcat$ is again as expected: the forest $f' = (g \cdot f)$ is given by the relationship that the $g(j)$th tree of $f'$ is the $j$th tree of $f$. The permutation $g' = g^f \in \grpd_V(n,n)$ has the following description. Identify $\{1,\ldots,n\}$ with the leaves of $f$ and with the leaves of $(g \cdot f)$. If $i$ is the $k$the leaf of the $j$th tree of $f$ then $g'(i)$ is the $k$th leaf of the $g(j)$th tree of $g \cdot f$, see Figure~\ref{fig:t_zappa_szep}.

At this point it becomes clear that working the actions as described above is virtually impossible. To obtain a more explicit algebraic description, we make use of the presentation of $\thomcat$. Property~\eqref{item:zappa-szep_product_acts_right} tells us that we know how any element of $\thomcat$ acts as soon as we know how the generators act and property~\eqref{item:zappa-szep_product_is_acted_right} tells us that we know how $\grpd_V$ acts on any element once we know how it acts on the generators of $\thomcat$. It therefore suffices to specify both actions for generators of $\thomcat$. Checking well-definedness then means to check various conditions coming from the relations in $\thomcat$.

So now we consider $g \in \grpd_V(m,m)$ and $\lambda_i^m \in \thomcat(m,m+1)$ and define the actions on each other. We start again with the easy case:
\begin{equation}
\label{eq:vcat_on_thomcat}
g \cdot \lambda_i = \lambda_{g(i)}\text{.}
\end{equation}
Working out $g^{\lambda_i}$ we have to distinguish four cases depending on the position of a point relative to $i$ and relative to $g(i)$:
\begin{equation}
\label{eq:thomcat_on_vcat}
g^{\lambda_i}(j) = \left\{
\begin{array}{ll}
g(j) & j \le i, g(j) \le g(i),\\
g(j - 1) & j > i, g(j-1) \le g(i),\\
g(j) + 1 & j \le i, g(j) > g(i),\\
g(j - 1) + 1 & j > i, g(j-1) > g(i)\text{.}
\end{array}
\right.
\end{equation}
Since $i = j$ if and only if $g(i) = g(j)$, the inequalities in the second and third case can be taken to be strict.

\begin{lemma}
\label{lem:thomcat_vcat_actions}
The formulas \eqref{eq:vcat_on_thomcat} and \eqref{eq:thomcat_on_vcat} define well-defined indirect product actions of $\thomcat$ and $\grpd_V$ on each other.
\end{lemma}

\begin{proof}
The conditions that involve only the action of $\grpd_V$, namely \eqref{item:zappa-szep_1_acts_left}, \eqref{item:zappa-szep_product_acts_left}, and \eqref{item:zappa-szep_1_is_acted_right} are clear. Condition \eqref{item:zappa-szep_1_acts_right} is defined to hold. Verifying conditions \eqref{item:zappa-szep_1_is_acted_left}, \eqref{item:zappa-szep_product_is_acted_left} on the $\lambda_i$ is straightforward, although in the second case tedious.

Conditions \eqref{item:zappa-szep_product_acts_right} and \eqref{item:zappa-szep_product_is_acted_right} should also be defined to hold, but in order for this to be well-defined, we need to check them on relations. That is, we need to verify that
\begin{align*}
(g^{\lambda_i})^{\lambda_j} = g^{\lambda_i \lambda_j} &= g^{\lambda_j\lambda_{i+1}} = (g^{\lambda_j})^{\lambda_{i+1}}\mathrlap{\quad\text{and}}\\
(g \cdot \lambda_i)(g_2^{\lambda_i} \cdot \lambda_j) = g \cdot (\lambda_i \lambda_j) &= g \cdot (\lambda_j \lambda_{i+1}) = (g \cdot \lambda_j)(g_2^{\lambda_j} \cdot \lambda_{i+1})
\end{align*}
for $j < i$. These are again not difficult but tedious and we skip them here. See \cite[Example~2.9]{witzel16a} for a detailed verification.
\end{proof}

Thus we can define $\Vcat = \thomcat \bowtie \grpd_V$.

\begin{lemma}
\label{lem:vcat_injective}
The action of $\thomcat$ on $\grpd_V$ defined by \eqref{eq:thomcat_on_vcat} is injective.
\end{lemma}

\begin{proof}
Since by definition $g^{\lambda_{i_1} \cdots \lambda_{i_n}} = (\ldots(g^{\lambda_{i_1}})\ldots)^{\lambda_{i_n}}$ we only need to check that the map $g \mapsto g^{\lambda_i}$ defined in \eqref{eq:vcat_on_thomcat} is injective. But $g$ can be recovered from $g^{\lambda_i}$ as follows. Let $\tau_i, \pi_i \colon \N \to \N$ be given by
\begin{align*}
\tau_i(j) &\defeq \left\{\begin{array}{ll}
j & j \le i\\
j+1 & j > i
\end{array}
\right.
&\pi_i(j) &\defeq\left\{\begin{array}{ll}
j & j \le i\\
j-1 & j > i
\end{array}
\right.
\end{align*}
Then $g(j) = \pi_i(g^{\lambda_i}(\tau_i(j)))$.
\end{proof}

Proposition~\ref{prop:thomcat_ore}, Corollary~\ref{cor:f_garside_family} and Proposition~\ref{prop:zappa-szep_conditions} now imply:

\begin{corollary}
\label{cor:vcat_righ_ore}
The category $\Vcat$ is right-Ore and admits a height function and a left-Garside family $\gars$ that is closed under factors such that $\gars(x,-)/\gars^\times$ is finite for every $x$.
\end{corollary}

The fundamental group $\pi_1(\Vcat,1)$ is \emph{Thompson's group $V$}.

\subsection{The braided Thompson groups}

The group $\mathit{BV}$, called \emph{braided $V$}, was introduced independently by Brin \cite{brin07} and Dehornoy \cite{dehornoy06}. We describe it using our framework which is similar to Brin's approach.

To define the categories underlying the braided Thompson groups, we define the groupoid $\grpd_{\mathit{BV}}$ to have as objects natural numbers, and to have morphisms $\grpd_{\mathit{BV}}(m,n) = \emptyset$ for $m \ne n$, and $\grpd_{\mathit{BV}}(n,n) = \Braid_n$. Note that the morphisms $\pi \colon \Braid_n \to \Sym_n$ define a morphism $\grpd_{\mathit{BV}} \to \grpd_V$ that we denote by $\pi$ as well. We want to define a indirect product $\thomcat \bowtie \grpd_{\mathit{BV}}$ and need to define actions of $\thomcat$ and $\grpd_{\mathit{BV}}$ on each other. Our guiding picture is Figure~\ref{fig:braided_v}.

\begin{figure}[htb]
\centering
\begin{tikzpicture}[scale=1,yscale=-1]
\begin{scope}
\draw (0,0) -- (0,1);
\draw (3,0) -- (3,1);
\draw
   (2,0) to [out=90, in=-90, looseness=1] (1,1);
\draw[white, line width=4pt]
   (1,0) to [out=90, in=-90, looseness=1] (2,1);
\draw
   (1,0) to [out=90, in=-90, looseness=1] (2,1);
\foreach \x in {0,1,2,3}
\node[end] at (\x,0) {};
\foreach \x in {0,1,2,3}
\node[end] at (\x,1) {};
\node[end] (root) at (1,1) {};
\node[end] (a) at (.5,1.5) {};
\node[end] (b) at (1.5,1.5) {};
\draw (a) -- (root) -- (b);
\foreach \x in {0,2,3}
\draw (\x,1.5) -- (\x,1);
\foreach \x in {0,2,3}
\node[end] at (\x,1.5) {};
\end{scope}
\node at (4,.75) {$=$};
\begin{scope}[shift={(5,0)}]
\draw (0,0) -- (0,2);
\draw (4,0) -- (4,2);
\draw
   (3,0) -- (3,1) to [out=90, in=-90, looseness=1] (2,2)
   (2,0) to [out=90, in=-90, looseness=1] (1,1) -- (1,2);
\draw[white, line width=4pt]
   (1,0) to [out=90, in=-90, looseness=1] (2,1) to [out=90, in=-90, looseness=1] (3,2);
\draw
   (1,0) to [out=90, in=-90, looseness=1] (2,1) to [out=90, in=-90, looseness=1] (3,2);
\foreach \x in {0,1,2,3,4}
\node[end] at (\x,0) {};
\foreach \x in {0,1,2,3,4}
\node[end] at (\x,2) {};
\node[end] (root) at (2.5,-.5) {};
\node[end] (a) at (2,0) {};
\node[end] (b) at (3,0) {};
\draw (a) -- (root) -- (b);
\foreach \x in {0,1,4}
\draw (\x,-.5) -- (\x,0);
\foreach \x in {0,1,4}
\node[end] at (\x,-.5) {};
\end{scope}
\end{tikzpicture}
\caption{Defining $\thomcat \bowtie \grpd_{\mathit{BV}}$.}
\label{fig:braided_v}
\end{figure}

We define the action of $\grpd_{\mathit{BV}}$ on $\thomcat$ simply as the action of $\grpd_V$ composed with $\pi$. In particular, $\sigma_i \cdot \lambda_i = \lambda_{i+1}$, $\sigma_i \cdot \lambda_{i+1} = \lambda_i$ and $\sigma_i \cdot \lambda_j = \lambda_j$ for $j \ne i,i+1$. The action of $\thomcat$ on $\grpd_{\mathit{BV}}$ we only define for generators acting on generators by
\[
\sigma_i^{\lambda_j} \defeq \left\{
\begin{array}{ll}
\sigma_{i+1} & j <i\\
\sigma_i \sigma_{i+1} & j = i\\
\sigma_{i+1} \sigma_i & j = i+1\\
\sigma_i & j > i+1\text{.}
\end{array}
\right.
\]

\begin{lemma}
\label{lem:bv_well_defined}
The formulas above define well-defined indirect product actions of $\thomcat$ and $\grpd_{\mathit{BV}}$ on each other.
\end{lemma}

In the proof we will use the fact that there is a set-theoretic splitting $\iota \colon \Sym_n \to \Braid_n$ that takes a reduced word $w(s_1, \ldots, s_{n-1})$ to the braid $w(\sigma_1,\ldots,\sigma_{n-1})$. This map is not multiplicative but if $\beta$ is a positive word (meaning involving no inverses) of length at most $3$ in the $\sigma_i$ then $\iota\pi(\beta) = \beta$.

\begin{proof}
As in the proof of Lemma~\ref{lem:thomcat_vcat_actions} most conditions hold by definition but we need to check well-definedness on relations. Namely
\small
\begin{align}
(\sigma_i \sigma_{i+1} \sigma_i) \cdot \lambda_k = \sigma_i \cdot (\sigma_{i+1} \cdot (\sigma_i \cdot \lambda_k)) &= \sigma_{i+1}\cdot (\sigma_i \cdot (\sigma_{i+1} \cdot \lambda_k)) = (\sigma_{i+1} \sigma_i \sigma_{i+1}) \cdot \lambda_k,\label{eq:bv_braid_acts}\\
(\sigma_i\sigma_{i+1}\sigma_i)^{\lambda_k} = \sigma_i^{(\sigma_{i+1}\sigma_i) \cdot \lambda_k}\sigma_{i+1}^{\sigma_i \cdot \lambda_k} \sigma_i^{\lambda_k}&= \sigma_{i+1}^{(\sigma_i\sigma_{i+1}) \cdot \lambda_k}\sigma_i^{\sigma_{i+1} \cdot \lambda_k} \sigma_{i+1}^{\lambda_k}=  (\sigma_{i+1}\sigma_i\sigma_{i+1})^{\lambda_k},\label{eq:bv_braid_is_acted}\\
(\sigma_i \sigma_j) \cdot \lambda_k = \sigma_i \cdot (\sigma_j \cdot \lambda_k) &= \sigma_i \cdot (\sigma_j \cdot \lambda_k) = (\sigma_j \sigma_i) \cdot \lambda_k,\label{eq:bv_commutator_acts}\\
(\sigma_i\sigma_j)^{\lambda_k} = \sigma_i^{\sigma_j \cdot \lambda_k} \sigma_j^{\lambda_k} &= \sigma_j^{\sigma_i \cdot \lambda_k} \sigma_i^{\lambda_k} = (\sigma_j\sigma_i)^{\lambda_k},\label{eq:bv_commutator_is_acted}\\
\sigma_i \cdot (\lambda_\ell \lambda_k) = (\sigma_i \cdot \lambda_\ell)(\sigma_i^{\lambda_\ell} \cdot \lambda_k) &= (\sigma_i \cdot \lambda_k)(\sigma_i^{\lambda_k} \cdot \lambda_{\ell+1}) = \sigma_i \cdot (\lambda_k \lambda_{\ell+1}), \mathrlap{\quad\text{and}}\label{eq:bv_split_is_acted}\\
\sigma_i^{\lambda_\ell \lambda_k} = (\sigma_i^{\lambda_\ell})^{\lambda_k} &= (\sigma_i^{\lambda_k})^{\lambda_{\ell+1}} = \sigma_i^{\lambda_k \lambda_{\ell+1}}\label{eq:bv_split_acts}
\end{align}
\normalsize
for $i-j \ge 2$, $\ell > k$.

Relations~\eqref{eq:bv_braid_acts} and~\eqref{eq:bv_commutator_acts} follow from Lemma~\ref{lem:thomcat_vcat_actions}. For the remaining relations note that $\pi(\beta^{\lambda_k}) = \pi(\beta)^{\lambda_k}$. Now~\eqref{eq:bv_split_is_acted} follows from Lemma~\ref{lem:thomcat_vcat_actions} as well because
\begin{equation}
\label{eq:thomcat_action_braids_permutations}
\pi(\sigma_i^{\lambda_\ell}) \cdot \lambda_{k} = \sigma_i^{\lambda_\ell} \cdot \lambda_{k}\quad\text{and}\quad\pi(\sigma_i^{\lambda_k}) \cdot \lambda_{\ell+1} = \sigma_i^{\lambda_k} \cdot \lambda_{\ell+1}\text{.}
\end{equation}
Relation~\eqref{eq:bv_split_acts} follows from Lemma~\ref{lem:thomcat_vcat_actions} by noting that both sides are positive words of length at most $3$ and applying $\iota$.

We verify \eqref{eq:bv_braid_is_acted} by distinguishing cases. The cases $k<i$ and $k > i+2$ are clear. If $k = i+1$ then the left hand side equals $(\sigma_{i+1}\sigma_i) \sigma_{i+2} (\sigma_{i+1}\sigma_i)$ and the right hand side equals $(\sigma_{i+1}\sigma_{i+2})\sigma_i(\sigma_{i+1}\sigma_{i+2})$. Both are equivalent through two braid relations with intermediate commutator relations. The cases $k = i$ and $k = i+2$ are symmetric and we only verify $k = i$. The left hand side equals $\sigma_i (\sigma_{i+1}\sigma_{i+2})(\sigma_i\sigma_{i+1})$ while the right hand side equals $(\sigma_{i+1}\sigma_{i+2})(\sigma_i\sigma_{i+1})\sigma_{i+2}$. Again these are equivalent through two braid relations with intermediate commutator relations.

Relation \eqref{eq:bv_commutator_is_acted} is left to the reader.
\end{proof}

For future reference we record \eqref{eq:thomcat_action_braids_permutations} which in the presence of Lemma~\ref{lem:bv_well_defined} can be formulated as:

\begin{observation}
\label{obs:bv_to_v_f_equiv}
The morphism $\pi \colon \grpd_{\mathit{BV}} \to \grpd_V$ is equivariant with respect to the $\thomcat$-action in the sense that
\[
\pi(\beta^f) = \pi(\beta)^f
\]
for $\beta \in \grpd_{\mathit{BV}}$ and $f \in \thomcat$.
\end{observation}

We define the category $\BVcat$ to be $\thomcat \bowtie \grpd_{\mathit{BV}}$ with the above indirect product actions.

\begin{lemma}
The action of $\grpd_{\mathit{BV}}$ on $\thomcat$ is injective.
\end{lemma}

\begin{proof}
We only need to check that $\beta \mapsto \beta^{\lambda_i}$ is injective. But $\beta$ can be recovered from $\beta^{\lambda_i}$ by removing the $(i+1)$st strand.
\end{proof}

\begin{corollary}
\label{cor:bvcat_righ_ore}
The category $\BVcat$ is right-Ore.
\end{corollary}

The fundamental group $\pi_1(\BVcat, 1)$ is the \emph{braided Thompson group $\mathit{BV}$}.

\medskip

\noindent
It is now easy to define braided versions of $T$ and $F$. We let $\grpd_{\mathit{BT}}$ and $\grpd_{\mathit{BF}}$ be the inverse image under $\pi \colon \grpd_{\mathit{BV}} \to \grpd_V$ of $\grpd_T$ and $\grpd_F$ respectively. Both of these act on $\thomcat$ by restricting the action of $\grpd_{\mathit{BV}}$, which is the same as to say that they act through $\pi$.

The action of $\thomcat$ of $\grpd_{\mathit{BV}}$ leaves $\grpd_{\mathit{BT}}$ and $\grpd_{\mathit{BF}}$ invariant and restricts to actions on these, thanks to Observation~\ref{obs:bv_to_v_f_equiv}: we know from Section~\ref{sec:tcat_vcat} that $\thomcat$ leaves $\grpd_{T}$ invariant and it is axiomatically required that it leaves the trivial groupoid invariant. Hence if $\beta \in \grpd_{\mathit{BT}}$ and $f \in \thomcat$ then $\pi(\beta^f) = \pi(\beta)^f \in \grpd_T$ so that $\beta^f \in \grpd_{\mathit{BT}}$ and an analogous reasoning applies for $\beta \in \grpd_{\mathit{BF}}$.

As a consequence we can define the categories $\BTcat = \thomcat \bowtie \grpd_{\mathit{BT}}$ and $\BFcat = \thomcat \bowtie \grpd_{\mathit{BF}}$, which are right-Ore. The group $\mathit{BF} = \pi_1(\BFcat,1)$ is called \emph{braided $F$} and was first introduced in \cite{brady08}. We call the group $\mathit{BT} = \pi_1(\BTcat,1)$ \emph{braided $T$}.

\begin{remark}
\label{rem:bt}
The group $\mathit{BT}$ was not introduced before for the following technical reason. Instead of our category $\BVcat$ Brin \cite{brin07} used a monoid that can be thought of as a category with a single object $\omega$ which represents countably infinitely many strands. This is possible because splitting one of countably infinitely many strands leads to countably infinitely many strands and because braid groups $\Braid_n$ are contained in a braid group $\varinjlim \Braid_n$ on infinitely many strands. A practical downside of that approach is that the group of fractions of that monoid is too big so one needs to describe which elements should be elements of $\mathit{BV}$. A formal downside is that groups like $\mathit{BT}$ or even $T$ cannot be described because $\Z/n\Z$ is not contained in $\Z/(n+1)\Z$ so that the needed limit does not exist.

Despite this formal problem, the main topological ingredient to establishing the finiteness properties of $\mathit{BT}$ has been verified in \cite[Section~3.4]{bux16}.
\end{remark}

\subsection{Groups arising from cloning systems}

In \cite{witzel16a} Zaremsky and the author have defined (filtered) cloning systems to be the data needed to define indirect product actions of $\thomcat$ and a groupoid on each other. Thus the groups considered there are by definition fundamental groups of categories $\thomcat \bowtie \grpd$ where $\grpd$ is a groupoid. However, the approach follows Brin \cite{brin07} to construct the groups as subgroups of an indirect product of monoids $\thomcat_\infty \bowtie \grpd_\infty$. As a consequence it has to deal with technical complications such as the notion of being \emph{properly graded}, as well as practical shortcomings such as being unable to construct (braided) $T$.

Our categorical approach removes the necessity that the groups $(G_n)_n$ fit into a directed system of groups and therefore the whole discussion goes through without that assumption. Thus a \emph{cloning system} is given by a sequence $(G_n)_{n \in \N}$ of groups, a sequence $(\rho_n)_{n \in \N} \colon G_n \to S_n$ of morphisms and a family of maps $(\kappa^n_k)_{k \le n} \colon G_n \to G_{n+1}$ such that the following hold for all $k \le n$, $k<\ell$, and $g,h\in G_n$:
 \begin{enumerate}[label={(CS\arabic*)}, ref={CS\arabic*}, leftmargin=*]
  \item $(gh) \kappa_k^n = (g)\kappa_{\rho(h)k}^n(h)\kappa_k^n$. \label{item:fcs_cloning_a_product}\hfill (Cloning a product)
  \item $\kappa_\ell^n \circ \kappa_k^{n+1} = \kappa_k^n \circ \kappa_{\ell+1}^{n+1}$.\label{item:fcs_product_of_clonings}\hfill (Product of clonings)
  \item $\rho_{n+1}((g)\kappa^n_k)(i) = (\rho_n(g)) \varsigma^n_k (i)$ for all $i\ne k,k+1$ \label{item:fcs_compatibility}\hfill(Compatibility)
 \end{enumerate}
Here $\varsigma^n_k$ describes the action of $\thomcat$ on $\grpd_V$ so that $(g)\varsigma^n_k(j) = g^{\lambda_k}(j)$ as in \eqref{eq:vcat_on_thomcat}.

Given a cloning system, a groupoid $\grpd$ is defined by setting $\grpd(m,n) = \emptyset$ if $m \ne n$ and setting $\grpd(n,n) = G_n$. Indirect product actions of $\thomcat$ and $\grpd$ on each other are defined by $g \cdot \lambda^n_k = \lambda^{n+1}_{\rho_n(g)k}$ and $g^{\lambda^n_k} = (g)\kappa^n_k$ for $g \in G_n$. The axioms \eqref{item:fcs_cloning_a_product}, \eqref{item:fcs_product_of_clonings}, \eqref{item:fcs_compatibility} ensure that these indeed define indirect product actions.

\subsection{The Higman--Thompson groups}

In total analogy to Section~\ref{sec:tcat_vcat} one can define $\Tcat_n = \thomcat_n \bowtie \grpd_T$ and $\Vcat_n = \thomcat_n \bowtie \grpd_V$. As mentioned in Section~{sec:fn} the category $\thomcat_n$ is not connected for $n > 2$ and neither are the categories $\Tcat_n$ and $\Vcat_n$. Thus it makes sense to define the groups
\begin{align*}
T_{n,r} &= \pi_1(\Tcat_n,r)\\
V_{n,r} &= \pi_1(\Vcat_n,r)
\end{align*}
and unlike the situation of $\thomcat_n$, these groups are indeed distinct for different $r$. They are the remaining \emph{Higman--Thompson groups}.

\subsection{Groups from graph rewriting systems}
\label{sec:graph_rewriting}

We now look at indirect products that do not involve $\thomcat$. The corresponding groups have been introduced and described in some detail in \cite{belk}. In this section, when we talk about graphs we will take their edges to be directed and allow multiple edges and loops. In particular, every vertex has an initial and a terminal vertex. The edge set of a graph $G$ is denoted $E(G)$ and the vertex set $V(G)$.

An \emph{edge replacement rule} $e \to R$ consists of a single directed edge $e$ and a finite graph $R$ that contains the two vertices of $e$ (but not $e$ itself). If $G$ is any graph and $\varepsilon$ is an edge of $G$, the edge replacement rule can be \emph{applied to $G$ at $\varepsilon$} by removing $\varepsilon$ and adding in a copy of $R$ while identifying the initial/terminal vertex of $\varepsilon$ with the initial/terminal vertex of $e$ in $R$. The resulting graph is denoted $G \lhd \varepsilon$. If $\delta$ is another edge of 
$G$, then it is also an edge of $G \lhd \varepsilon$ and so the replacement rule can be applied to $G \lhd \varepsilon$ at $\delta$. We regard $G \lhd \varepsilon \lhd \delta$ and $G \lhd \delta \lhd \varepsilon$ as the same graph.

The vagueness inherent in the last sentence can be remedied by declaring that a graph obtained from $G$ by applying the edge replacement rule (possibly many times) has as edges words in $E(G) \times E(R)^*$ and as vertices words in $V(G) \cup (E(G) \times E(R)^* \times V(R))$. For example, the graph $G \lhd \varepsilon \lhd \delta$ would have edges $\zeta \in E(G) \smallsetminus \{\varepsilon, \delta\}$ and $\varepsilon\xi$, $\delta\xi$ for $\xi \in E(R)$ and vertices $v \in V(G)$ as well as $\varepsilon w$ and $\delta w$ for $w \in V(R)$.

For every edge replacement rule $e \to R$ we define a category $\rewcat_{e \to R}$ whose objects are finite graphs. In order for the category to be small we will take the graphs to have vertices and edges coming from a fixed countable set, which in addition is closed under attaching words in $E(R)$ and $V(R)$. The category is presented by having generators
\[
\lambda^G_\varepsilon \in \rewcat_{e \to R}(G, G\lhd \varepsilon)\quad \text{for}\quad G\text{ a graph and }\varepsilon \text{ an edge of }G
\]
subject to the relations
\begin{equation}
\label{eq:rewriting_relation}
\lambda^G_\delta \lambda^{G \lhd \delta}_\varepsilon = \lambda^{G}_\varepsilon \lambda^{G \lhd \varepsilon}_\delta \quad \text{for}\quad G\text{ a graph and }\delta, \varepsilon \text{ distinct edges of }G\text{.}
\end{equation}

\begin{lemma}
For any edge replacement rule $e \to R$ the category $\rewcat_{e \to R}$ is right-Ore.
\end{lemma}

\begin{proof}
Thanks to the relations \eqref{eq:rewriting_relation} a morphism $\lambda_{\varepsilon_1} \ldots \lambda_{\varepsilon_k}$ in $\rewcat_{e \to R}$ is uniquely determined by its source, its target, and the set $\{\varepsilon_1, \ldots, \varepsilon_k\}$. The claim now follows by taking differences and unions of these sets of edges.
\end{proof}

As in previous sections, the second ingredient will be a groupoid. Its 
definition does not depend on the edge replacement rule, except possibly for the 
foundational issues of choosing universal sets of vertices and edges. We define 
$\grpd_\text{graph}$ to have as objects finite graphs and as morphisms 
isomorphisms of graphs.

We define actions of $\rewcat_{e \to R}$ and $\grpd_\text{graph}$ on each other as follows. If $g \colon G \to G'$ is an isomorphism of graphs and $\varepsilon \in E(G)$ is an edge then
\[
g \cdot \lambda^G_{\varepsilon} = \lambda^{G'}_{g(\varepsilon)}
\]
and $g^{\lambda_{\varepsilon}}$ is the isomorphism $G \lhd \varepsilon \to G' \lhd g(\varepsilon)$ that takes $\delta$ to $g(\delta)$ for $\delta \in E(G) \smallsetminus \{\varepsilon\}$ and that takes $\varepsilon\zeta$ to $g(\varepsilon)\zeta$ for $\zeta \in V(R) \cup E(R)$. The following is easy to verify:

\begin{observation}
The actions of $\rewcat_{e \to R}$ and $\grpd_\text{graph}$ on each other defined above are well-defined indirect product actions. The action of $\rewcat_{e \to R}$ on $\grpd_\text{graph}$ is injective.
\end{observation}

As a consequence we obtain a right-Ore category $\rewgpcat_{e \to R} \defeq \rewcat_{e \to R} \bowtie \grpd_\text{graph}$ and for every finite graph $G$ we obtain a potential group $\pi_1(\rewgpcat_{e \to R}, G)$.

\begin{example}
If we consider the edge replacement rule
\begin{center}
\begin{tikzpicture}[scale=1, thick]
\node at (-1,1) {$e=$};
\node at (4,1) {$=L_2$};
\begin{scope}[every node/.style={circle, fill, inner sep=2pt}]
\node[label=left:{$v$}] (lv) at (0,0) {};
\node[label=left:{$w$}] (lw) at (0,2) {};

\node[label=left:{$v$}] (v) at (2.5,0) {};
\node[label=left:{$w$}] (w) at (2.5,2) {};
\node (4) at (2.5,1) {};
\end{scope}
\draw[->-] (lv) to (lw);

\node at (1,1) {$\to$};

\draw[->-] (v) to[left] (4);
\draw[->-] (4) to[left] (w);
\end{tikzpicture}
\end{center}
and take $L_1$ to be the graph consisting of a single edge then then $\pi_1(\rewgpcat_{e \to L_2}, L_1)$ is isomorphic to $F$. Similarly, if $C_1$ is the graph consisting of a single loop then $\pi_1(\rewgpcat_{e \to L_2}, C_1)$ is isomorphic to $T$. Finally, $V$ arises as $\pi_1(\rewgpcat_{e \to D_2}, L_1)$ where the rule $e \to D_2$ replaces an edge by two disconnected edges.
\end{example}

Various fundamental groups of categories arising from graph rewriting systems are described in \cite{belk}. Here we will only mention the Basilica Thompson group introduced by them in \cite{belk15}.

We consider the replacement rule
\begin{center}
\begin{tikzpicture}[scale=1, thick]
\node at (-1,1) {$e=$};
\node at (4.5,1) {$=R$};
\begin{scope}[every node/.style={circle, fill, inner sep=2pt}]
\node[label=left:{$v$}] (lv) at (0,0) {};
\node[label=left:{$w$}] (lw) at (0,2) {};

\node[label=left:{$v$}] (v) at (2.5,0) {};
\node[label=left:{$w$}] (w) at (2.5,2) {};
\node (4) at (2.5,1) {};
\end{scope}
\draw[->-] (lv) to (lw);

\node at (1,1) {$\to$};

\draw[->-] (v) to[left] (4);
\draw[->-] (4) to[out=-\loopangle, in=\loopangle, min distance=1cm, looseness=5,right] (4);
\draw[->-] (4) to[left] (w);
\end{tikzpicture}
\end{center}
and the graph
\begin{center}
\begin{tikzpicture}[scale=2, thick]
\node at (-1.2,0) {$G=$};
\begin{scope}[every node/.style={circle, fill, inner sep=2pt}]
\node (x) at (0,0) {};
\node (y) at (1,0) {};
\end{scope}
\begin{scope}
\path (x) edge[->-,out=180-\loopangle, in = 180+\loopangle, min distance = 1cm, looseness=5, left] (x);
\path (y) edge[->-,out=0-\loopangle, in = 0+\loopangle, min distance = 1cm, looseness=5, right] node{.} (y);
\path (x) edge[->-,out=0-\picangle, in = 180+\picangle, below] (y);
\path (y) edge[->-,out=180-\picangle, in = 0+\picangle, above] (x);
\end{scope}
\end{tikzpicture}
\end{center}

The \emph{Basilica Thompson group} is $T_B \defeq \pi_1(\rewgpcat_{e \to R}, G)$.


\section{Examples: Finiteness properties}
\label{sec:finiteness_properties_examples}

In this final section we give various examples of applications of Theorem~\ref{thm:generic_proof} and Corollary~\ref{cor:generic_proof}. In most cases these finiteness properties are known and the proofs involve proving that certain complexes are highly connected. We will see that these complexes always coincide with the complexes $\abs{E(x)}$. As a consequence the connectivity statement from the literature together with Theorem~\ref{thm:generic_proof} gives the result.

\subsection{Finiteness properties of Thompson's groups}

We start with the categories $\Tcat$ and $\Vcat$. The conditions needed to apply the results from Section~\ref{sec:finiteness_properties} have been verified in Corollary~\ref{cor:tcat_righ_ore} and Corollary~\ref{cor:vcat_righ_ore}.

In order to apply Corollary~\ref{cor:generic_proof} two more things are left to verify: that automorphism groups are of type $F_\infty$ and that the connectivity of the simplicial complexes $\abs{E(n)}$ goes to infinity with $n$. The groups $\thomcat(n,n) = \{1\}$, $\Tcat(n,n) = \Z/n\Z$ and $\Vcat(n,n) = \Sym_n$ are all finite and therefore of type $F_\infty$.

In order to describe the complexes $E(n)$, we need to talk about further graphs. The \emph{cyclic graph} is denoted $C_n$, it has the same edges as $L_n$ and additionally $\{1,n\}$. The \emph{complete graph} $K_n$ has all edges $\{i,j\}$, $1 \le i < j \le n$. We describe the complexes $E(n)$ in the case of $\Vcat$ and leave $\Tcat$ to the reader.

\begin{lemma}
\label{lem:link_match_cyclic}
The poset $E_\Tcat(n)$ is isomorphic to $\match(C_n)$.
\end{lemma}

\begin{lemma}
\label{lem:link_match_complete}
There is a poset-morphism $E_\Vcat(n) \to \match(K_n)$ whose fibers over $k$-simplices are $k$-spheres.
\end{lemma}

\begin{proof}
Every element of $(\eltry \bowtie \grpd_V)(-,n)$ can be written as a product $fg$ of an elementary forest $f \in \eltry(-,n)$ and a permutation $g \in \grpd_V(n,n)$. By definition the vertices of $E(n)$ are these products modulo multiplication by permutations from the left. As in Lemma~\ref{lem:link_match_linear} an elementary forest can be interpreted as a matching on $L_n$. Under this correspondence, the group $\grpd_V(n,n) = \Sym_n$ acts on the vertices of $L_n$ and the permutations from the left act on components of the matching. Thus elements of $(\eltry \bowtie \grpd_V)(-,n)$ can be described by matchings on the linear graph on $g^{-1}(1), \ldots, g^{-1}(n)$ modulo reordering the components of the matching.

The possibility of reordering the vertices of the matching means that any two elements of $\{1,\ldots,n\}$ can be connected and so we obtain a map $\abs{E(n)} \to \match(K_n)$ to the matching complex of the \emph{complete} graph on $\{1,\ldots,n\}$. This map is clearly surjective.

It is not injective because in $E(n)$ the order of two matched vertices matters while in $\match(K_n)$ it does not. For example, $\lambda_i$ and $\lambda_i (i\ i+1)$ map to the same vertex in $\match(K_n)$. As a result the fiber over a $k$-simplex is a join of $k+1$ many $0$-spheres, i.e.\ a $k$-sphere.
\end{proof}

The fact that the morphism in Lemma~\ref{lem:link_match_complete} is not an isomorphism means that we have to do one extra step, namely to apply the following result by Quillen \cite[Theorem~9.1]{quillen73}. Rather than giving the general formulation for posets we restrict to face posets of ($n$-skeleta of) simplicial complexes, to save us some notation.

\begin{theorem}
\label{thm:quillen}
Let $n \in \N$ and let $f \colon X \to Y$ be a simplicial map. Assume that $Y$ is $(n-1)$-connected and that for every $k$-simplex $\sigma$ of $Y$ the link $\lk \sigma$ is $(n -\dim \sigma-2)$-connected and the fiber $\abs{f^{-1}(\sigma)}$ is $(k-1)$-connected. Then $X$ is $(n-1)$-connected.
\end{theorem}

\begin{theorem}
Thompson's groups $T$, and $V$ are of type $F_\infty$.
\end{theorem}

\begin{proof}
Using Corollary~\ref{cor:generic_proof} we need to show that the connectivity of the complexes $\abs{E(n)}$ goes to infinity as $n$ goes to infinity. We work with the simplicial complexes $E(n)$ instead. In the case of $T$ the complexes are matching complexes by Lemma~\ref{lem:link_match_cyclic} whose connectivity goes to infinity by Lemma~\ref{lem:match_connectivity}. In the case of $V$ the complexes map to matching complexes with good fibers by Lemma~\ref{lem:link_match_complete}. Noting that the link of a $k$-simplex in $\match(K_n)$ is isomorphic to  $\match(K_{n-2(k+1)})$, we can apply Theorem~\ref{thm:quillen} to see that the connectivity of $E_\Vcat$ goes to infinity as well.
\end{proof}

\subsection{Finiteness properties of braided Thompson groups}

We have alreadt seen that $\BFcat$, $\BTcat$, and $\BVcat$ are right-Ore. That they admit a height function and a left-Garside family follows via Proposition~\ref{prop:zappa-szep_conditions}, just as it did for $\Tcat$ and $\Vcat$. The braid groups $\BVcat^\times(n,n) = \grpd_{\mathit{BV}}(n,n)$ are of type $F$ by Corollary~\ref{cor:garside_complex} (and hence of type $F_\infty$). Consequently the finite index subgroups of pure braids $\BFcat^\times(n,n)$ and of cyclically-permuting braids $\BTcat^\times(n,n)$ are of type $F$ as well.

It remains to understand the complexes $\abs{E(n)}$. For that purpose, we will want to think of braid groups as mapping class groups. Let $D$ be a closed disc with $n$ punctures $p_1, \ldots, p_n$ which we can think of as distinguished points in the interior of $D$. The mapping class group of the $n$-punctured disc is
\[
\Homeo^+(D \smallsetminus \{p_1,\ldots, p_n\}, \partial D)/\Homeo^+_0(D \smallsetminus \{p_1,\ldots, p_n\}, \partial D)
\]
where $\Homeo^+(D \smallsetminus \{p_1,\ldots, p_n\}, \partial D)$ is the group of orientation-preserving homeomorphisms of $D \smallsetminus \{p_1,\ldots, p_n\}$ that fix $\partial D$ and $\Homeo^+_0(D \smallsetminus \{p_1,\ldots, p_n\}, \partial D)$ is the subgroup of homeomorphisms that are isotopic to the identity. It is well-known that the mapping class group of the $n$-punctured disc is isomorphic to the braid group, see for example \cite{kassel08}.

With this description in place, we can start to look at the complexes $\abs{E(n)}$. Let $fg \in E(n)$ with $f \in \eltry(-,n)$ and $g \in \grpd_{\mathit{BV}}(n,n)$. Regard the $n$ punctures $p_1, \ldots, p_n$ as the vertices of an $L_n$ embedded into $D$. As we have seen before, $f$ corresponds to a matching $M_f$ on $L_n$ which we now regard as a disjoint selection of the fixed arcs connecting pairs of adjacent punctures. The element $g$, regarded as a mapping class, acts on $M_f$ and we obtain a set $M_fg$ of disjoint arcs connecting some pairs of punctures. Such a collection of arcs is called an \emph{arc matching} in \cite{bux16}. Note that if $f \in \eltry(k,n)$ so that the arc matching consists of $n-k$ arcs, then removing the arcs from the punctured disc results in a $k$-punctured disc. The action of $\grpd_{\mathit{BV}}(k,k)$ from the left is just the action of the mapping class group of that $k$-punctured disc and in particular does nothing to $M_f$.

For a subgraph $\Gamma$ of $K_n$ the \emph{arc matching complex} $\arcmatch(\Gamma)$ is the simplicial complex whose $k$-simplices are sets of pairwise disjoint arcs connecting punctures with the condition that an arc can only connect two punctures if they are connected by an edge in $\Gamma$.

\begin{proposition}
\label{prop:link_arc_match}
There surjective morphisms of simplicial complexes
\begin{enumerate}
\item $E_\BFcat(n) \to \arcmatch(L_n)$\label{item:arc_map_BF}
\item $E_\BTcat(n) \to \arcmatch(C_n)$\label{item:arc_map_BT}
\item $E_\BVcat(n) \to \arcmatch(K_n)$\label{item:arc_map_BV}
\end{enumerate}
whose fiber over any $k$-simplex is the join of $k$ countable infinite discrete sets.
\end{proposition}

\begin{proof}
The product $fg \in E(n)$ is taken to the arc matching $M_fg$ as described above. Since $\grpd_{\mathit{BF}}(n,n)$ takes every puncture to itself, the map \eqref{item:arc_map_BF} maps onto $\arcmatch(L_n)$. Similarly, since $\grpd_{\mathit{BT}}(n,n)$ cyclically permutes the punctures, the map \eqref{item:arc_map_BT} maps into $\arcmatch(C_n)$. Surjectivity is clear.

To describe the fibers consider a disc $D'$ containing $p_i$ and $p_{i+1}$ but none of the other punctures and let $\beta$ be a braid that is arbitrary inside $D'$ but trivial outside. Then $\lambda_i \beta$ maps to the same arc (= vertex of $\arcmatch(K_n)$) irrespective of $\beta$. Thus the fiber over this vertex is the mapping class group of $D' \smallsetminus \{p_i, p_{i+1}\}$ in the case of $\BVcat$ and is the pure braid group of $D' \smallsetminus \{p_i,p_{i+1}\}$ in the cases of $\BFcat$ and $\BTcat$. In either case it is a countable infinite discrete set.
\end{proof}

The connectivity properties of arc matching complexes have been studied in \cite{bux16}. We summarize Theorem~3.8, Corollary~3.11, and the remark in Section~3.4 in the following theorem. It applies to arc matching complexes not only on disks but on arbitrary surfaces with (possibly empty) boundary.

\begin{theorem}
\label{thm:arc_matching_connectivity}
\begin{enumerate}
\item $\arcmatch(K_n)$ is $(\nu(n)-1)$-connected,
\item $\arcmatch(C_n)$ is $(\eta(n-1)-1)$-connected,
\item $\arcmatch(L_n)$ is $(\eta(n)-1)$-connected,
\end{enumerate}
where $\nu(n) = \floor{\frac{n-1}{3}}$ and $\eta(n) = \floor{\frac{n-1}{4}}$.
\end{theorem}

\begin{theorem}
\label{thm:braided_finiteness_properties}
The braided Thompson groups $\mathit{BF}$, $\mathit{BT}$, and $\mathit{BV}$ are of type $F_\infty$.
\end{theorem}

\begin{proof}
We want to apply Corollary~\ref{cor:generic_proof}. By Proposition~\ref{prop:link_arc_match} the complexes $E(n)$ map onto arc matching complexes and we want to apply Theorem~\ref{thm:quillen}. To do so, we need to observe that the link of a $k+1$-simplex on an arc matching complex on a surface with $n$ punctures is an arc matching complex with $n-2k$ punctures where the $k$ arcs connecting two punctures have been turned into boundary components. Putting these results together shows that the connectivity properties of $E(n)$ go to infinity with $n$ by Theorem~\ref{thm:arc_matching_connectivity}.
\end{proof}

\subsection{Absence of finiteness properties}

Theorem~\ref{thm:generic_proof} gives a way to prove that certain groups are of type $F_n$. If the group is not of type $F_n$ one of the hypotheses fails. We will now discuss to which extent the construction is (un-)helpful in proving that the group is not of type $F_n$ depending on which hypothesis fails.

In the first case the groups $\cat^\times(x,x)$ are not of type $F_n$ (even for $\rho(x)$ large). In this case the general part of Brown's criterion, Theorem~\ref{thm:browns_criterion_negative}, cannot be applied. Thus the whole construction from Section~\ref{sec:proof_scheme} is useless for showing that $\pi_1(\grpd_\cat,*)$ is not of type $F_n$. An example of this case are the groups $\mathcal{T}(B_*(\mathcal{O}_S))$ treated in \cite[Theorem~8.12]{witzel16a}. The proof redoes part of the proof that the groups $\cat^\times(x,x)$, which are the groups in $B_n(\mathcal{O}_S)$ in this case, are not of type $F_n = F_{\abs{S}}$.

In the second case the complexes $E(x)$ are not (even asymptotically) $(n-1)$-connected. In this case Brown's criterion, Theorem~\ref{thm:browns_criterion_negative}, can in principle be applied, but not by using just Morse theory. An example of this case is the Basilica Thompson group from Section~\ref{sec:graph_rewriting} which is not finitely presented \cite{witzel16}, so $n=2$. A morphism in $\rewgpcat_{e \to R} = \rewcat_{e \to R} \bowtie \grpd_{\text{graph}}$ is declared to be elementary if there are edges $\{e_1,\ldots,e_k\}$ of $G$ such that $f = \lambda_{e_1}\ldots\lambda_{e_k}$. The function $\rho \colon \rewcat_{e \to R} \to \N$ is the number of edges of a graph. The basepoint $*$ is the Basilica graph $G$.

The connectivity-assumption of Theorem~\ref{thm:generic_proof} is violated because the $\rewgpcat_{e \to R}$-com\-po\-nent of $G$ contains graphs $H$ with arbitrarily many edges for which $E(H)$ is not simply connected. Examples of such graphs are illustrated in Figure~\ref{fig:bad_graph}. In these examples $E(H)$ has four vertices, two vertices $v_{ll}$, $v_{ul}$ corresponding to the loops on the left and two vertices $v_{lr}$ and $v_{ur}$ corresponding to the loops on the right. The left vertices are connected to the right vertices but not to each other and neither are the right vertices. Thus $E(H)$ is a circle $v_{ll}, v_{lr}, v_{ul}, v_{ur}$ and is not simply connected.

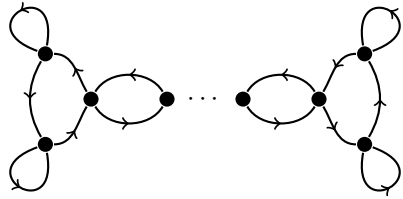
\begin{figure}[htb]
\begin{center}
\begin{tikzpicture}[thick]
\begin{scope}[every node/.style={circle, fill, inner sep=2pt}]
\node (oo) at (1,0) {};
\node (o) at (0,0) {};
\node (p) at (-.6,.6) {};
\node (q) at (-.6,-.6) {};
\node (kk) at (2,0) {};
\node (k) at (3,0) {};
\node (m) at (3.6,.6) {};
\node (l) at (3.6,-.6) {};
\end{scope}
\node at (1.5,0) () {$\cdots$};
\begin{scope}
\path (o) edge[->-,out=0-\picangle, in = 180+\picangle, below] (oo);
\path (oo) edge[->-,out=180-\picangle, in = 0+\picangle, below] (o);
\path (o) edge[->-,out=180-\picangle, in = 60-\picangle, above right] (p);
\path (p) edge[->-,out=120-\loopangle, in = 120+\loopangle, min distance = 1cm, looseness=5, above left] (p);
\path (p) edge[->-,out=180+\picangle, in = 180-\picangle, left] (q);
\path (q) edge[->-,out=240-\loopangle, in = 240+\loopangle, min distance = 1cm, looseness=5, below left] (q);
\path (q) edge[->-,out=300+\picangle, in = 180+\picangle, below right] (o);

\path (kk) edge[->-,out=0-\picangle, in = 180+\picangle, below] (k);
\path (k) edge[->-,out=180-\picangle, in = 0+\picangle, below] (kk);
\path (k) edge[->-,out=0-\picangle, in = 240-\picangle, below left] (l);
\path (l) edge[->-,out=300-\loopangle, in = 300+\loopangle, min distance = 1cm, looseness=5, below right] (l);
\path (l) edge[->-,out=0+\picangle, in = 0-\picangle, right] (m);
\path (m) edge[->-,out=60-\loopangle, in = 60+\loopangle, min distance = 1cm, looseness=5, above right] (m);
\path (m) edge[->-,out=120+\picangle, in = 0+\picangle, above left] (k);
\end{scope}
\end{tikzpicture}
\end{center}
\caption{Arbitrarily large graphs $H$ with $E(H)$ not simply connected.}
\label{fig:bad_graph}
\end{figure}

Looking into the proof of Theorem~\ref{thm:generic_proof} we can compare directly what the non-simple connectedness of $E(H)$ tells us and what is needed to apply Brown's criterion (Theorem~\ref{thm:browns_criterion_negative}) in order to prove that the group is not of type $F_n$. To apply Theorem~\ref{thm:browns_criterion_negative}, one needs to show that for every $m$ there is an arbitrarily large $n$ such that passing from $X_{\rho<m}$ to $X_{\rho<n+1}$ a non-trivial $1$-sphere in $X_{\rho < m}$ is filled in. The assumption that $E(H)$ is not simply connected for $\rho(H) = n$ translates via the Morse argument to the statement that when passing from $X_{\rho <n}$ to $X_{\rho < n+1}$ either a non-trivial $1$-sphere in $X_{\rho < n}$ is filled in, or a non-trivial $2$-sphere is created. The proof in \cite{witzel16} that the Basilica-Thompson group $T_B$ is not finitely presented, therefore needs to rule out the second possibility and also show that the $1$-sphere that is filled in was non-trivial already in $X_{\rho<m}$.


\bibliographystyle{amsalpha}
\bibliography{ore_cats}

\newcommand{\etalchar}[1]{$^{#1}$}
\providecommand{\bysame}{\leavevmode\hbox to3em{\hrulefill}\thinspace}
\providecommand{\MR}{\relax\ifhmode\unskip\space\fi MR }
\providecommand{\MRhref}[2]{%
  \href{http://www.ams.org/mathscinet-getitem?mr=#1}{#2}
}
\providecommand{\href}[2]{#2}
\begin{thebibliography}{MPMN16}

\bibitem[Art25]{artin25}
Emil Artin, \emph{{T}heorie der {Z}öpfe.}, Abh. Math. Semin. Univ. Hamb.
  \textbf{4} (1925), 47--72.

\bibitem[BB97]{bestvina97}
Mladen Bestvina and Noel Brady, \emph{Morse theory and finiteness properties of
  groups}, Invent. Math. \textbf{129} (1997), no.~3, 445--470. \MR{1465330}

\bibitem[BBCS08]{brady08}
Tom Brady, Jos\'e Burillo, Sean Cleary, and Melanie Stein, \emph{Pure braid
  subgroups of braided {T}hompson's groups}, Publ. Mat. \textbf{52} (2008),
  no.~1, 57--89. \MR{2384840}

\bibitem[Bel04]{belk04}
James~Michael Belk, \emph{Thompson's group ${F}$}, Ph.D. thesis, Cornell
  University, 2004, arXiv:0708.3609.

\bibitem[Bes03]{bessis03}
David Bessis, \emph{The dual braid monoid}, Ann. Sci. \'Ecole Norm. Sup. (4)
  \textbf{36} (2003), no.~5, 647--683. \MR{2032983}

\bibitem[BF]{belk}
James Belk and Bradley Forrest, \emph{Rearrangement groups of fractals},
  arXiv:1510.03133.

\bibitem[BF15]{belk15}
\bysame, \emph{A {T}hompson group for the {B}asilica}, Groups Geom. Dyn.
  \textbf{9} (2015), no.~4, 975--1000. \MR{3428407}

\bibitem[BFM{\etalchar{+}}16]{bux16}
Kai-Uwe Bux, Martin~G. Fluch, Marco Marschler, Stefan Witzel, and Matthew C.~B.
  Zaremsky, \emph{The braided {T}hompson's groups are of type {$\rm
  F_\infty$}}, J. Reine Angew. Math. \textbf{718} (2016), 59--101, With an
  appendix by Zaremsky. \MR{3545879}

\bibitem[BG84]{brown84}
Kenneth~S. Brown and Ross Geoghegan, \emph{An infinite-dimensional torsion-free
  {${\rm FP}_{\infty }$} group}, Invent. Math. \textbf{77} (1984), no.~2,
  367--381. \MR{752825}

\bibitem[Bir74]{birman74}
Joan~S. Birman, \emph{Braids, links, and mapping class groups}, Princeton
  University Press, Princeton, N.J.; University of Tokyo Press, Tokyo, 1974,
  Annals of Mathematics Studies, No. 82. \MR{0375281}

\bibitem[BKL98]{birman98}
Joan Birman, Ki~Hyoung Ko, and Sang~Jin Lee, \emph{A new approach to the word
  and conjugacy problems in the braid groups}, Adv. Math. \textbf{139} (1998),
  no.~2, 322--353. \MR{1654165}

\bibitem[BLV{\v{Z}}94]{bjoerner94}
A.~Bj{\"o}rner, L.~Lov{\'a}sz, S.~T. Vre{\'c}ica, and R.~T.
  {\v{Z}}ivaljevi{\'c}, \emph{Chessboard complexes and matching complexes}, J.
  London Math. Soc. (2) \textbf{49} (1994), no.~1, 25--39.

\bibitem[Bra01]{brady01}
Thomas Brady, \emph{A partial order on the symmetric group and new
  {$K(\pi,1)$}'s for the braid groups}, Adv. Math. \textbf{161} (2001), no.~1,
  20--40. \MR{1857934}

\bibitem[Bri05]{brin05}
Matthew~G. Brin, \emph{On the {Z}appa-{S}z\'ep product}, Comm. Algebra
  \textbf{33} (2005), no.~2, 393--424. \MR{2124335}

\bibitem[Bri07]{brin07}
\bysame, \emph{The algebra of strand splitting. {I}. {A} braided version of
  {T}hompson's group {$V$}}, J. Group Theory \textbf{10} (2007), no.~6,
  757--788. \MR{2364825}

\bibitem[Bro87]{brown87}
Kenneth~S. Brown, \emph{Finiteness properties of groups}, Proceedings of the
  {N}orthwestern conference on cohomology of groups ({E}vanston, {I}ll., 1985),
  vol.~44, 1987, pp.~45--75. \MR{885095}

\bibitem[BS72]{brieskorn72}
Egbert Brieskorn and Kyoji Saito, \emph{Artin-{G}ruppen und
  {C}oxeter-{G}ruppen}, Invent. Math. \textbf{17} (1972), 245--271.
  \MR{0323910}

\bibitem[CFP96]{cannon96}
J.~W. Cannon, W.~J. Floyd, and W.~R. Parry, \emph{Introductory notes on
  {R}ichard {T}hompson's groups}, Enseign. Math. (2) \textbf{42} (1996),
  no.~3-4, 215--256. \MR{1426438}

\bibitem[CMW04]{charney04}
R.~Charney, J.~Meier, and K.~Whittlesey, \emph{Bestvina's normal form complex
  and the homology of {G}arside groups}, Geom. Dedicata \textbf{105} (2004),
  171--188. \MR{2057250}

\bibitem[DDG{\etalchar{+}}15]{dehornoy15}
Patrick {Dehornoy}, Fran\c{c}ois {Digne}, Eddy {Godelle}, Daan {Kramer}, and
  Jean {Michel}, \emph{Foundations of {G}arside theory.}, European Mathematical
  Society (EMS), 2015 (English).

\bibitem[Deh06]{dehornoy06}
Patrick Dehornoy, \emph{The group of parenthesized braids}, Adv. Math.
  \textbf{205} (2006), no.~2, 354--409. \MR{2258261}

\bibitem[Far03]{farley03}
Daniel~S. Farley, \emph{Finiteness and {$\rm CAT(0)$} properties of diagram
  groups}, Topology \textbf{42} (2003), no.~5, 1065--1082. \MR{1978047}

\bibitem[FMWZ13]{fluch13}
Martin~G. Fluch, Marco Marschler, Stefan Witzel, and Matthew C.~B. Zaremsky,
  \emph{The {B}rin-{T}hompson groups {$sV$} are of type {$\text{F}_\infty$}},
  Pacific J. Math. \textbf{266} (2013), no.~2, 283--295.

\bibitem[Gar69]{garside69}
F.~A. Garside, \emph{The braid group and other groups}, Quart. J. Math. Oxford
  Ser. (2) \textbf{20} (1969), 235--254. \MR{0248801}

\bibitem[Hat02]{hatcher02}
Allen Hatcher, \emph{Algebraic topology}, Cambridge University Press,
  Cambridge, 2002. \MR{1867354}

\bibitem[Hig74]{higman74}
Graham Higman, \emph{Finitely presented infinite simple groups}, Department of
  Pure Mathematics, Department of Mathematics, I.A.S. Australian National
  University, Canberra, 1974, Notes on Pure Mathematics, No. 8 (1974).
  \MR{0376874}

\bibitem[Koz08]{kozlov08}
Dmitry Kozlov, \emph{Combinatorial algebraic topology}, Springer, 2008.

\bibitem[KT08]{kassel08}
Christian Kassel and Vladimir Turaev, \emph{Braid groups}, Graduate Texts in
  Mathematics, vol. 247, Springer, New York, 2008, With the graphical
  assistance of Olivier Dodane. \MR{2435235}

\bibitem[MPMN16]{martinezperez16}
Conchita Mart{\'{\i}}nez-P{\'e}rez, Francesco Matucci, and Brita E.~A.
  Nucinkis, \emph{Cohomological finiteness conditions and centralisers in
  generalisations of {T}hompson's group {$V$}}, Forum Math. \textbf{28} (2016),
  no.~5, 909--921. \MR{3543701}

\bibitem[Qui73]{quillen73}
Daniel Quillen, \emph{Higher algebraic {$K$}-theory. {I}}, Algebraic
  {$K$}-theory, {I}: {H}igher {$K$}-theories ({P}roc. {C}onf., {B}attelle
  {M}emorial {I}nst., {S}eattle, {W}ash., 1972), Springer, Berlin, 1973,
  pp.~85--147. Lecture Notes in Math., Vol. 341. \MR{0338129}

\bibitem[Qui78]{quillen78}
\bysame, \emph{Homotopy properties of the poset of nontrivial {$p$}-subgroups
  of a group}, Adv. in Math. \textbf{28} (1978), no.~2, 101--128. \MR{493916}

\bibitem[Ste92]{stein92}
Melanie Stein, \emph{Groups of piecewise linear homeomorphisms}, Trans. Amer.
  Math. Soc. \textbf{332} (1992), no.~2, 477--514. \MR{1094555}

\bibitem[Wal65]{wall65}
C.T.C. Wall, \emph{Finiteness conditions for {${\rm CW}$}-complexes}, Ann.\ of
  Math.\ (2) \textbf{81} (1965), 56--69.

\bibitem[Wal66]{wall66}
\bysame, \emph{Finiteness conditions for {${\rm CW}$} complexes. {II}}, Proc.\
  Roy.\ Soc.\ Ser.\ A \textbf{295} (1966), 129--139.

\bibitem[Wit16]{witzel16b}
Stefan Witzel, \emph{Finiteness properties of thompson groups}, Bielefeld
  University, 2016, Habilitation Thesis.

\bibitem[WZa]{witzel}
Stefan Witzel and Matthew~C.B. Zaremsky, \emph{The ${\Sigma}$-invariants of
  {T}hompson's group ${F}$ via {M}orse theory}, arXiv:1501.06682, to appear in
  the proceedings of Conference ``Topological Methods in Group Theory''.

\bibitem[WZb]{witzel16}
\bysame, \emph{The {B}asilica {T}hompson group is not finitely presented},
  arXiv:1603.01150.

\bibitem[WZc]{witzel16a}
\bysame, \emph{Thompson groups for systems of groups, and their finiteness
  properties}, arXiv:1405.5491, to appear in Groups, Geometry, and Dynamics.

\end{thebibliography}

\end{document}